\documentclass{amsart}
\usepackage[utf8]{inputenc}
\usepackage[english]{babel}

\usepackage[margin=2.5cm]{geometry}
\usepackage{enumerate}
\usepackage{float}
\usepackage[final]{graphicx}
\usepackage{epstopdf} 

\usepackage{tabularx}
\newcolumntype{C}[1]{>{\centering\arraybackslash}m{#1}}

\usepackage{caption}
\usepackage{subcaption}
\usepackage{enumitem}
\usepackage{comment}
\usepackage{soul}

\usepackage{tikz-cd}
\usepackage{amsmath, multirow}
\usepackage{amsfonts}
\usepackage{amssymb}
\usepackage{amsthm}
\usepackage{hyperref}
\usepackage{cleveref}

\DeclareCaptionSubType[alph]{figure}
\usepackage[T1]{fontenc}
\usepackage{mathrsfs}
\usepackage{mathtools}

\usepackage{tikz}
\usetikzlibrary{decorations.markings}
\usetikzlibrary{decorations.pathreplacing}
\usetikzlibrary{arrows,shapes,positioning,patterns}
\usetikzlibrary{knots}
\tikzstyle{none}=[inner sep=0pt]
\pgfdeclarelayer{edgelayer}
\pgfdeclarelayer{nodelayer}
\pgfsetlayers{edgelayer,nodelayer,main}
\usepackage{url} 
\usepackage[all]{xy}

\newcommand{\bN}{\mathbb N}

\newcommand{\R}{\mathcal R}

\newcommand{\bZ}{\mathbb Z}

\newcommand{\x}{\times}

\theoremstyle{plain}
\newtheorem{thm}{Theorem}[section]
\newtheorem{prop}[thm]{Proposition}
\newtheorem{lem}[thm]{Lemma}
\newtheorem{cor}[thm]{Corollary}
\theoremstyle{remark}
\newtheorem{rem}[thm]{Remark}

\newtheorem{example}[thm]{Example}

\theoremstyle{definition}
\newtheorem{defn}[thm]{Definition}
\newtheorem{conj}[thm]{Conjecture}

\newtheorem*{crit*}{Criterion~B}

\definecolor{aquamarine}{rgb}{0.5, 1.0, 0.83}
\definecolor{princetonorange}{rgb}{1.0, 0.56, 0.0}
\definecolor{caribbeangreen}{rgb}{0.0, 0.8, 0.6}
\definecolor{bunired}{rgb}{0.8, 0.0, 0.0}
\definecolor{cdgreen}{rgb}{0.0, 0.42, 0.24}
\definecolor{lavender(floral)}{rgb}{0.71, 0.49, 0.86}
\definecolor{bluedefrance}{rgb}{0.19, 0.55, 0.91}
\definecolor{iris}{rgb}{0.35, 0.31, 0.81}
\definecolor{darkgreen}{rgb}{0.33, 0.42, 0.18}

\newcommand{\tC}{{\tt C}}

\newcommand{\tG}{{\tt G}}
\newcommand{\tH}{{\tt H}}
\newcommand{\tK}{{\tt K}}
\newcommand{\tI}{{\tt I}}
\newcommand{\tD}{{\tt D}}
\newcommand{\tT}{{\tt T}}
\newcommand{\tR}{{\tt R}}
\newcommand{\tS}{{\tt S}}
\newcommand{\tP}{{\tt P}}
\newcommand{\tL}{{\tt L}}
\newcommand{\tA}{{\tt A}}
\newcommand{\tM}{{\tt M}}

\newcommand{\Comp}[2]{C_{#2}(#1)}

\title{On the homotopy type of multipath complexes}
\author{Luigi Caputi}
\author{Carlo Collari}
\author{Sabino Di Trani}
\author{Jason P. Smith}

\begin{document}

\maketitle

\begin{abstract}
A multipath in a directed graph is a disjoint union of paths. The multipath complex of a directed graph $\tG$ is the simplicial complex whose faces are the multipaths of $\tG$.  We compute the Euler characteristic, and associated generating function, of the multipath complex for some families of graphs, including transitive tournaments and complete bipartite graphs. Then, we compute the homotopy type of multipath complexes of linear graphs, polygons, small grids and transitive tournaments. We show that they are all contractible or wedges of spheres. We introduce a new technique for decomposing directed graphs into dynamical regions, which allows us to simplify the homotopy computations.
\end{abstract}


\section{Introduction}

Simplicial complexes associated to monotone properties of (directed) graphs are central objects in both combinatorics and topology (cf.~\cite{Jonsson}) with interesting and deep connections with other areas of mathematics -- see, e.g.~\cite{Vassiliev1993,MR2022345, PaoliniSalvetti}. 
Particularly relevant examples of simplicial complexes arising from monotone properties are the well-known matching complex and its relatives, the independence complex and the flag complex.  In this work, we focus  on multipath complexes, which are also related (albeit differently from independence and flag complexes) to matching complexes~\cite[Section~4]{SpriSecondo}. 
The simplices of the multipath complex are called multipaths~\cite{turner}, and are disjoint unions of directed paths.
Multipath complexes appeared in \cite{Omega} -- denoted therein by $\Omega(\tG)$ -- and were studied for~$\tG = \tK_n$, the complete directed graph, in virtue of their relation to symmetric homology of algebras~\cite{AultFed,AultNoFed}.
A first step in a systematic investigation of topological and combinatorial properties of  multipath complexes was taken in~\cite{secondo}, and was motivated by homological questions~\cite{primo}.
In this paper, we continue the study of the combinatorial and topological properties of multipath complexes of directed graphs. 
More precisely, we provide both qualitative and quantitative information about their homotopy type.

One of the main results in~\cite{secondo} asserts that the homology of multipath complexes can be fairly rich; namely, it can be supported in arbitrarily high degree, and can be of arbitrarily high rank. 
A qualitative measure of this complexity is the (reduced) Euler characteristic. We compute the Euler characteristic, as well as generating functions, for infinite families of directed graphs such as transitive tournaments and complete bipartite graphs -- this is developed in Section~\ref{sec:generating}.
It is worth noting that, for transitive tournaments, the Euler characteristic of the associated multipath complexes can be expressed in terms of Stirling numbers of the second kind, and that the associated generating function is doubly exponential. This is qualitatively different from the generating function of the Euler characteristic of matching complexes of complete graphs -- cf.~\cite[Table~10.2]{Jonsson} -- which is exponential. 
Instead, the Euler characteristic of the multipath complex for complete bipartite graphs with alternating orientation is  the Euler characteristic of the chessboard complex
-- previously investigated in~\cite{MatchingChess}.

In the second part of this work we focus on the explicit description of the homotopy type of multipath complexes. The general question about what kind of simplicial complexes can be realised as multipath complexes remains open. Here we employ topological tools and use combinatorial techniques to identify the homotopy type of the multipath complex for some infinite families: linear graphs, polygons, small grids, and transitive tournaments. We prove that multipath complexes associated to these families are either contractible or wedges of spheres. 
To simplify the computation of the homotopy type of multipath complexes we introduce a decomposition of directed graphs into dynamical regions (cf.~Definition~\ref{def:dynamicalregion}). Intuitively, dynamical regions are determined by the behaviour of flows in directed graphs; 
when moving from a vertex of this region, while following the orientation, one either stays in the region or goes out without coming back. Minimal dynamical regions are called dynamical modules. We prove the following;

\begin{thm}
Let $\tG$ be a directed graph. Then, there is a unique (up to re-ordering) decomposition of $\tG$ in to dynamical modules $\tM_1,...,\tM_{k}$, and we have a homotopy equivalence
 \[X(\tG) \simeq X(\tM_1) \ast \cdots \ast X(\tM_k) \ ,\]
where $X(-)$ denotes the multipath complex. Furthermore, the above decomposition can be found algorithmically.
\end{thm}

The decomposition into dynamical modules, for certain families of directed graphs, might be trivial. This is the case, for instance, of transitive tournaments. Then, the computation of the homotopy type of the associated multipath complexes needs different methods. Borrowing techniques from combinatorial topology, we show that the multipath complex of transitive tournaments on $n\geq 3$ vertices is  homotopy equivalent to a wedge of spheres (Theorem~\ref{thm:cliqueshomotopy}). 

This result is in sharp contrast to what happens with the homotopy type of the matching complex for complete graphs; the latter is not known in general, but it is known that its homology has torsion in specific degrees -- see, e.g.~\cite{torsion, MR2470116, MR2731551}. For stable 
 dynamical regions (cf.~Definition~\ref{def:dynamicalregion}), multipath complexes and matching complexes are isomorphic (see Lemma~\ref{lem:alternatingmultipath}), hence also the multipath complex can have torsion -- see~\cite[Proposition~4.5]{SpriSecondo}. We conjecture that, for a dynamical module $\tM$, if  the multipath complex~$X(\tM)$ has torsion, then $\tM$ is stable.

The computations of the Euler characteristics presented in this work use the custom package \textsc{path\_poset}, publicly available at~\cite{pathposet}. To compute homology this package was combined with SageMath~\cite{sagemath}.

\subsection*{Acknowledgements}

LC acknowledges support from the \'{E}cole Polytechnique F\'{e}d\'{e}rale de Lausanne via a collaboration agreement with the University of Aberdeen. CC is supported by the MIUR-PRIN project 2017JZ2SW5.
LC and CC acknowledge partial support from the Heilbronn Small Grants Scheme. SDT is partially supported by the “National Group for Algebraic and Geometric Structures, and their Applications” (GNSAGA – INdAM). 
LC warmly thanks Ran Levi for the useful conversations, motivation and support.
The authors are grateful to Paolo Lisca and Roberto Pagaria for their comments on the drafts of this paper.

\section{Basic notions}
In this section we recall some basic notions  needed {throughout}. A (finite) undirected \emph{graph} \tG~is a pair of (finite) sets~$(V,E)$ consisting of a set~$V$ of \emph{vertices}, and a set~$E$ of~\emph{edges} given by unordered pairs of distinct vertices of~$\tG$. All graphs are assumed to be simple{, that is, do not contain loops or multiedges}.
We also consider directed graphs, or \emph{digraphs}, a (finite) digraph~$\tG$ is a pair of (finite) sets $(V(\tG),E(\tG))$, such that $E(\tG)$ is a set of ordered pairs of distinct vertices.
Given an edge $e = (v, w)$ of $E(\tG)$ we call the vertex $v$ the \emph{source}~of~$e$, denoted $s(e)$, while the vertex $w$ is the \emph{target} of~$e$, denoted $t(e)$.
An \emph{orientation} on an undirected graph is the choice of a source and of a target for each edge. 
An undirected graph~$\tG$ can be turned into a directed graph by orienting each edge of $\tG$ in both directions; vice versa, given a directed graph, we can consider the underlying \emph{simple} undirected graph obtained by forgetting the directions of the edges, and merging any multiedges.

A \emph{subgraph}~$\tH$ of a (directed) graph~$\tG$ is a (directed) graph such that $V(\tH)\subseteq V(\tG)$ and $E(\tH)\subseteq E(\tG)$; if $\tH$ is a subgraph of $\tG$, we  write $\tH \leq \tG$.
If $\tH \leq \tG$ and $\tH\neq \tG$ we say that $\tH$ is a \emph{proper subgraph} of $\tG$, and we write~$\tH< \tG$. We say that $\tH$ is an \emph{induced subgraph}  of a (directed) graph~$\tG$ if for any pair of vertices $v,w$ in $\tH$, if $e$ is an edge in $\tG$ between $v$ and $w$, then $e$ is also an edge of $\tH$.  Furthermore, if $\tH\leq \tG$ and $V(\tH) = V(\tG)$ we  say that $\tH$ is a \emph{spanning subgraph} of~$\tG$. Two edges in an undirected graph~$\tG$ are called \emph{adjacent} if they share a common vertex.  
A  \emph{simple path} in a digraph~$\tG$  is a sequence of edges $e_1,...,e_n$ of $\tG$ such that~$s(e_{i+1})=t(e_i)$ for $i=1,\dots,n-1$, and no vertex is encountered twice, i.e.~if $s(e_i) = s(e_j)$ or $t(e_i) = t(e_j)$, then $i=j$, and is not a cycle, i.e.~$s(e_1)\neq t(e_n)$ -- cf.~Figure~\ref{fig:nstep}.

		
		
\begin{figure}[h]
\centering
	\begin{tikzpicture}[baseline=(current bounding box.center),line join = round, line cap = round]
		\tikzstyle{point}=[circle,thick,draw=black,fill=black,inner sep=0pt,minimum width=2pt,minimum height=2pt]
		\tikzstyle{arc}=[shorten >= 8pt,shorten <= 8pt,->, thick]
		\def\c{8}\def\d{.5}
		\node[above] (v0) at (0-\c,0+\d) {$v_0$};\draw[fill] (0-\c,0+\d)  circle (.05);
		\node[above] (v1) at (1.5-\c,0+\d) {$v_1$};\draw[fill] (1.5-\c,0+\d)  circle (.05);
		\node[above] (v2) at (3-\c,0+\d) {$v_{2}$};\draw[fill] (3-\c,0+\d)  circle (.05);
		\node[above] (v3) at (4.5-\c,0+\d) {$v_{3}$};\draw[fill] (4.5-\c,0+\d)  circle (.05);
		
		\draw[thick, bunired, -latex] (0.15-\c,0+\d) -- (1.35-\c,0+\d);
		\draw[thick, bunired, -latex] (1.65-\c,0+\d) -- (2.85-\c,0+\d);
		\draw[thick, bunired, -latex] (3.15-\c,0+\d) -- (4.35-\c,0+\d);
		
		\node (e1) at (0,0) {};
		\node (e2) at (2,0) {};
		\node (e3) at (60:2) {};
		\draw[pattern=north west lines, pattern color=bluedefrance,thick] (e1.center) -- (e2.center) -- (e3.center) -- (e1.center);
		\node[circle,fill=bunired,scale=0.5] at (e1) {};\node[below] at (e1) {$(v_0,v_1)$};
        \node[circle,fill=bunired,scale=0.5] at (e2) {};\node[below] at (e2) {$(v_1,v_2)$};
        \node[circle,fill=bunired,scale=0.5] at (e3) {};\node[above] at (e3) {$(v_2,v_3)$};

    \end{tikzpicture}
	\begin{tikzpicture}
        \def\y{1}\def\x{3}
        \tikzstyle{sp}=[circle,fill=black,scale=0.3]
        \tikzstyle{se}=[thick, bunired, -latex]
        \node (p123) at (0*\x,3*\y){\begin{tikzpicture}[scale=0.5]\node[sp] (0) at (0,0){};\node[sp] (1) at (1,0){};\node[sp] (2) at (2,0){};\node[sp] (3) at (3,0){};\draw[se] (0) -- (1);\draw[se] (1) -- (2);\draw[se] (2) -- (3);\end{tikzpicture}};
        
        \node (p12) at (-1*\x,2*\y){\begin{tikzpicture}[scale=0.5]\node[sp] (0) at (0,0){};\node[sp] (1) at (1,0){};\node[sp] (2) at (2,0){};\node[sp] (3) at (3,0){};\draw[se] (0) -- (1);\draw[se] (1) -- (2);\end{tikzpicture}};
        \node (p13) at (0*\x,2*\y){\begin{tikzpicture}[scale=0.5]\node[sp] (0) at (0,0){};\node[sp] (1) at (1,0){};\node[sp] (2) at (2,0){};\node[sp] (3) at (3,0){};\draw[se] (0) -- (1);\draw[se] (2) -- (3);\end{tikzpicture}};
        \node (p23) at (1*\x,2*\y){\begin{tikzpicture}[scale=0.5]\node[sp] (0) at (0,0){};\node[sp] (1) at (1,0){};\node[sp] (2) at (2,0){};\node[sp] (3) at (3,0){};\draw[se] (1) -- (2); \draw[se] (2) -- (3);\end{tikzpicture}};
        \node (p1) at (-1*\x,1*\y){\begin{tikzpicture}[scale=0.5]\node[sp] (0) at (0,0){};\node[sp] (1) at (1,0){};\node[sp] (2) at (2,0){};\node[sp] (3) at (3,0){};\draw[se] (0) -- (1);\end{tikzpicture}};
        \node (p2) at (0*\x,1*\y){\begin{tikzpicture}[scale=0.5]\node[sp] (0) at (0,0){};\node[sp] (1) at (1,0){};\node[sp] (2) at (2,0){};\node[sp] (3) at (3,0){};\draw[se] (1) -- (2);\end{tikzpicture}};
        \node (p3) at (1*\x,1*\y){\begin{tikzpicture}[scale=0.5]\node[sp] (0) at (0,0){};\node[sp] (1) at (1,0){};\node[sp] (2) at (2,0){};\node[sp] (3) at (3,0){};\draw[se] (2) -- (3);\end{tikzpicture}};
        \node (p0) at (0*\x,0*\y){\begin{tikzpicture}[scale=0.5]\node[sp] (0) at (0,0){};\node[sp] (1) at (1,0){};\node[sp] (2) at (2,0){};\node[sp] (3) at (3,0){};\end{tikzpicture}};
        \draw[thick, dotted] (p123) -- (p12) -- (p1) -- (p0);
        \draw[thick,dotted] (p123) -- (p23) -- (p2) -- (p0);
        \draw[thick,dotted] (p123) -- (p13) -- (p3) -- (p0);
        \draw[thick,dotted] (p12) -- (p2);
        \draw[thick,dotted] (p23) -- (p3);
        \draw[thick,dotted] (p13) -- (p1);
	\end{tikzpicture}

	\caption{The coherently oriented linear graph $\tI_3$ (top left), the multipath complex $X(\tI_3)$ (top right), and the path poset $P(\tI_3)$ (bottom).}
	\label{fig:nstep}
\end{figure}
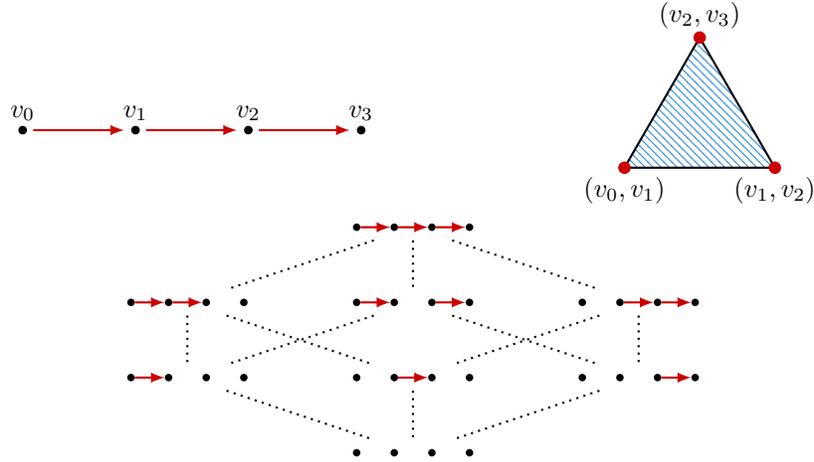

We are interested in disjoint sets of simple paths; following~\cite{turner}, we call them multipaths:

\begin{defn}\label{def:multipaths}
A \emph{multipath} of a digraph~$\tG$ is a spanning subgraph such that each connected component is either a vertex or a simple path. The \emph{length} of a multipath is the number of its edges. 
\end{defn}

The set of multipaths of $\tG$ has a natural partially ordered structure: the \emph{path poset} of $\tG$ is the poset $(P(\tG),<)$, that is, the set of multipaths of $\tG$ (including the multipath with no edges) ordered by the relation of ``being a subgraph''. 
Note that the underlying set of $P(\tG)$ is given by all disjoint unions of simple paths -- as opposed to all connected paths, as in~\cite[Section~3.1]{https://doi.org/10.48550/arxiv.2205.03730}.
To the path poset we can associate a simplicial complex, which we call the multipath complex  -- cf.~\cite[Definition~6.4]{secondo}:

\begin{defn}\label{def:pathcomplx}
For a digraph $\tG$, the \emph{multipath complex} $X(\tG)$ is the simplicial complex whose face poset (augmented to include the empty simplex~$\emptyset$) is the path poset $P(\tG)$. 
\end{defn}

Since being a multipath is a monotone property of digraphs~(for a description of monotone properties, see \cite{bjorner_cdg}, and the references therein), it follows that $X(\tG)$ is a well-defined simplicial complex.
The following is straightforward:

\begin{example}[{\cite[Example~6.12]{secondo}}]
Consider the coherently oriented linear graph $\tI_n$ -- see Figure~\ref{fig:nstep} for an example of $I_3$.
The path poset~$(P(\tI_n), < )$ is isomorphic to the Boolean poset~$\mathbb{B}(n)$. Thence, the associated multipath complex is an $(n-1)$-dimensional simplex.
Consider the coherently oriented polygonal graph $\tP_n$ with $n$ edges, obtained from $\tI_{n}$ by identifying the vertices $v_0$ and $v_{n}$. Then, the path poset $(P(\tP_n),<)$ is isomorphic to the Boolean poset $\mathbb{B}(n)$ minus its maximum, and the corresponding multipath complex is a $(n-2)$-dimensional sphere.
\end{example}

Another class of directed graphs important to us is given by the dandelion graphs:
\begin{defn}\label{dandelion}
Let $\tD_{n,m}$ be the digraph on $(n+m+1)$ vertices  and $(m+n)$ edges defined as follows:
\begin{enumerate}
    \item $V(\tD_{n,m}) = \{ v_{0}, w_{1} ,\dots, w_{n}, x_{1},\dots, x_{m} \}$;
    \item $E(\tD_{n,m}) = \{ (w_i,v_0), (v_{0},x_j) \mid i=1,\dots,n; j=1,\dots,m \}$.
\end{enumerate}
The digraph $\tD_{n,m}$ is called a \emph{dandelion graph} -- cf.~Figure~\ref{fig:nmgraph}. A dandelion graph of the form $\tD_{n,0}$ (resp.~ $\tD_{0,m}$) is called \emph{sink graph} (resp.~\emph{source graph}).
\end{defn}

\begin{figure}[h]
\centering
	\begin{tikzpicture}[scale=0.4][baseline=(current bounding box.center)]
		\tikzstyle{point}=[circle,thick,draw=black,fill=black,inner sep=0pt,minimum width=2pt,minimum height=2pt]
		\tikzstyle{arc}=[shorten >= 8pt,shorten <= 8pt,->, thick]
		\def\x{15}
		\node (v0) at (0-\x,0) {};
		\node[above] at (v0) {$v_0$};
		\draw[fill] (v0)  circle (.05);
		\node (w1) at (-2-\x,1.5) {};
		\node[above] at (w1) {$w_1$};
		\draw[fill] (w1)  circle (.05);
		\node  (w2) at (-2-\x,0) { };
		\node[left]  at (w2) {$w_{2}$};
		\draw[fill] (w2)  circle (.05);
		\node  (w3) at (-2-\x,-1.5) { };
		\node[below]  at (w3) {$w_{3}$};
		\draw[fill] (w3)  circle (.05);
		
		\node  (x1) at (2-\x,1.3) { };
		\node[right]  at (x1) {$x_1$};
		\draw[fill] (x1)  circle (.05);
		
		\node  (x2) at (2-\x,-1.3) { };
		\node[right]  at (x2) {$x_2$};
		\draw[fill] (x2)  circle (.05);

		\draw[thick, bunired, -latex] (w1) -- (v0);
		\draw[thick, bunired, -latex] (w2) -- (v0);
		\draw[thick, bunired, -latex] (w3) -- (v0);
		\draw[thick, bunired, -latex] (v0) -- (x1);
		\draw[thick, bunired, -latex] (v0) -- (x2);
		
		\node[minimum size=2cm,regular polygon,regular polygon sides=5] (p) {};
        \draw[draw=black,thick] (p.corner 1) -- (p.corner 2) -- (p.corner 4) -- (p.corner 5) -- (p.corner 3) -- (p.corner 2);
        \draw[draw=black,thick] (p.corner 5) -- (p.corner 1);
		\node[circle,fill=bunired,scale=0.5] at (p.corner 1) {};\node[above] at (p.corner 1) {\small$(w_1,v_0)$};
        \node[circle,fill=bunired,scale=0.5] at (p.corner 2) {};\node[left] at (p.corner 2) {\small$(v_0,x_1)$};
        \node[circle,fill=bunired,scale=0.5] at (p.corner 3) {};\node[below] at (p.corner 3) {\small$(w_2,v_0)$};
        \node[circle,fill=bunired,scale=0.5] at (p.corner 4) {};\node[below] at (p.corner 4) {\small$(w_3,v_0)$};
        \node[circle,fill=bunired,scale=0.5] at (p.corner 5) {};\node[right] at (p.corner 5) {\small$(v_0,x_2)$};
	\end{tikzpicture}
	\vskip 30pt
	\begin{tikzpicture}[scale=0.9]
        \def\y{2}\def\x{3}
        \tikzstyle{sp}=[circle,fill=black,scale=0.3]
        \tikzstyle{se}=[thick, bunired, -latex]

        \node (a1) at (-2.5*\x,2*\y){\begin{tikzpicture}[scale=0.3]\node[sp] (v0) at (0,0){};\node[sp] (x1) at (2,.75){};\node[sp] (x2) at (2,-.75){};\node[sp] (w1) at (-2,1){};\node[sp] (w2) at (-2,0){};\node[sp] (w3) at (-2,-1){};\draw[se] (v0) -- (x1);\draw[se] (w1) -- (v0);\end{tikzpicture}};
        \node (a2) at (-1.5*\x,2*\y){\begin{tikzpicture}[scale=0.3]\node[sp] (v0) at (0,0){};\node[sp] (x1) at (2,.75){};\node[sp] (x2) at (2,-.75){};\node[sp] (w1) at (-2,1){};\node[sp] (w2) at (-2,0){};\node[sp] (w3) at (-2,-1){};\draw[se] (v0) -- (x1);\draw[se] (w2) -- (v0);\end{tikzpicture}};
        \node (a3) at (-0.5*\x,2*\y){\begin{tikzpicture}[scale=0.3]\node[sp] (v0) at (0,0){};\node[sp] (x1) at (2,.75){};\node[sp] (x2) at (2,-.75){};\node[sp] (w1) at (-2,1){};\node[sp] (w2) at (-2,0){};\node[sp] (w3) at (-2,-1){};\draw[se] (v0) -- (x1);\draw[se] (w3) -- (v0);\end{tikzpicture}};
        \node (a4) at (0.5*\x,2*\y){\begin{tikzpicture}[scale=0.3]\node[sp] (v0) at (0,0){};\node[sp] (x1) at (2,.75){};\node[sp] (x2) at (2,-.75){};\node[sp] (w1) at (-2,1){};\node[sp] (w2) at (-2,0){};\node[sp] (w3) at (-2,-1){};\draw[se] (v0) -- (x2);\draw[se] (w1) -- (v0);\end{tikzpicture}};
        \node (a5) at (1.5*\x,2*\y){\begin{tikzpicture}[scale=0.3]\node[sp] (v0) at (0,0){};\node[sp] (x1) at (2,.75){};\node[sp] (x2) at (2,-.75){};\node[sp] (w1) at (-2,1){};\node[sp] (w2) at (-2,0){};\node[sp] (w3) at (-2,-1){};\draw[se] (v0) -- (x2);\draw[se] (w2) -- (v0);\end{tikzpicture}};
        \node (a6) at (2.5*\x,2*\y){\begin{tikzpicture}[scale=0.3]\node[sp] (v0) at (0,0){};\node[sp] (x1) at (2,.75){};\node[sp] (x2) at (2,-.75){};\node[sp] (w1) at (-2,1){};\node[sp] (w2) at (-2,0){};\node[sp] (w3) at (-2,-1){};\draw[se] (v0) -- (x2);\draw[se] (w3) -- (v0);\end{tikzpicture}};
        
        \node (b1) at (-2*\x,1*\y){\begin{tikzpicture}[scale=0.3]\node[sp] (v0) at (0,0){};\node[sp] (x1) at (2,.75){};\node[sp] (x2) at (2,-.75){};\node[sp] (w1) at (-2,1){};\node[sp] (w2) at (-2,0){};\node[sp] (w3) at (-2,-1){};\draw[se] (v0) -- (x1);\end{tikzpicture}};
        \node (b2) at (-1*\x,1*\y){\begin{tikzpicture}[scale=0.3]\node[sp] (v0) at (0,0){};\node[sp] (x1) at (2,.75){};\node[sp] (x2) at (2,-.75){};\node[sp] (w1) at (-2,1){};\node[sp] (w2) at (-2,0){};\node[sp] (w3) at (-2,-1){};\draw[se] (w1) -- (v0);\end{tikzpicture}};
        \node (b3) at (0*\x,1*\y){\begin{tikzpicture}[scale=0.3]\node[sp] (v0) at (0,0){};\node[sp] (x1) at (2,.75){};\node[sp] (x2) at (2,-.75){};\node[sp] (w1) at (-2,1){};\node[sp] (w2) at (-2,0){};\node[sp] (w3) at (-2,-1){};\draw[se] (w2) -- (v0);\end{tikzpicture}};
        \node (b4) at (1*\x,1*\y){\begin{tikzpicture}[scale=0.3]\node[sp] (v0) at (0,0){};\node[sp] (x1) at (2,.75){};\node[sp] (x2) at (2,-.75){};\node[sp] (w1) at (-2,1){};\node[sp] (w2) at (-2,0){};\node[sp] (w3) at (-2,-1){};\draw[se] (w3) -- (v0);\end{tikzpicture}};
        \node (b5) at (2*\x,1*\y){\begin{tikzpicture}[scale=0.3]\node[sp] (v0) at (0,0){};\node[sp] (x1) at (2,.75){};\node[sp] (x2) at (2,-.75){};\node[sp] (w1) at (-2,1){};\node[sp] (w2) at (-2,0){};\node[sp] (w3) at (-2,-1){};\draw[se] (v0) -- (x2);\end{tikzpicture}};
        
        \node (bot) at (0*\x,0*\y){\begin{tikzpicture}[scale=0.3]\node[sp] (v0) at (0,0){};\node[sp] (x1) at (2,.75){};\node[sp] (x2) at (2,-.75){};\node[sp] (w1) at (-2,1){};\node[sp] (w2) at (-2,0){};\node[sp] (w3) at (-2,-1){};\end{tikzpicture}};

        \draw[thick, dotted] (a1) -- (b2) -- (bot);
        \draw[thick, dotted] (a2) -- (b3) -- (bot);
        \draw[thick, dotted] (a3) -- (b4) -- (bot);
        \draw[thick, dotted] (a4) -- (b2) -- (bot);
        \draw[thick, dotted] (a5) -- (b3) -- (bot);
        \draw[thick, dotted] (a6) -- (b4) -- (bot);
        \draw[thick, dotted] (a1) -- (b1) -- (bot);
        \draw[thick, dotted] (a2) -- (b1);
        \draw[thick, dotted] (a3) -- (b1);
        \draw[thick, dotted] (a4) -- (b5) -- (bot);
        \draw[thick, dotted] (a5) -- (b5);
        \draw[thick, dotted] (a6) -- (b5);
	\end{tikzpicture}
	\caption{The dandelion graph $\tD_{3,2}$ (top left), its multipath complex $X(\tD_{3,2})$ (top right), and its path poset $P(\tD_{3,2})$ (bottom).}
	\label{fig:nmgraph}
\end{figure}
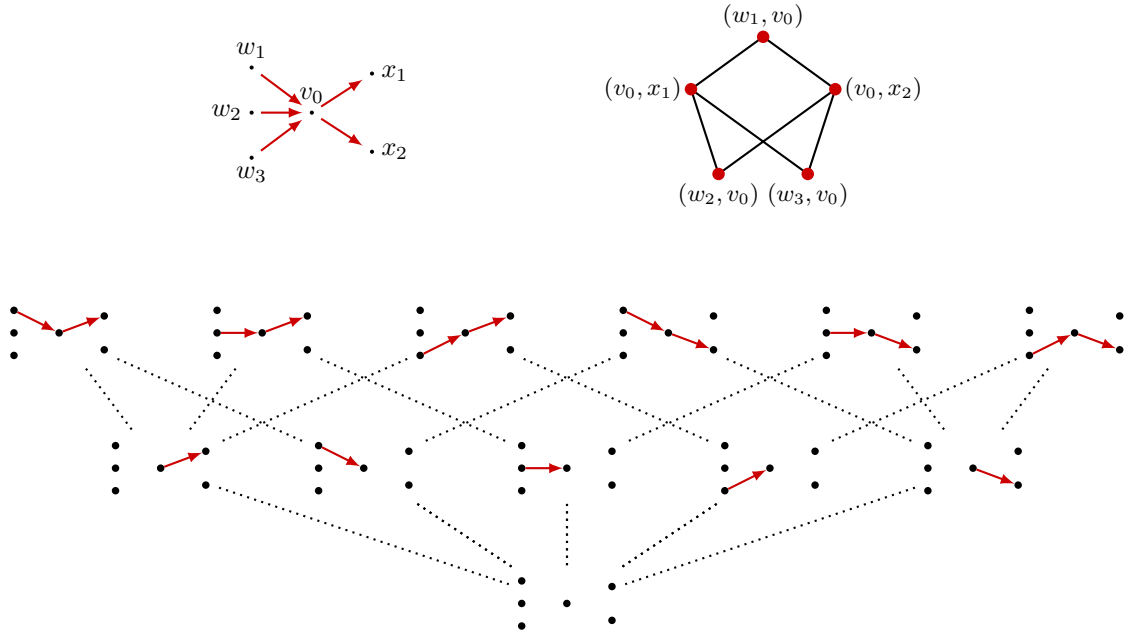

\begin{example}
The multipath complex $X(\tD_{n,m})$ of the dandelion graph $\tD_{n,m}$ is homotopy equivalent to the wedge of~$(n-1)(m-1)$ copies of the $1$-dimensional sphere if $n,m>1$ -- see Figure~\ref{fig:nmgraph} and \cite[Example~6.13]{secondo}. If either $n$ or $m$ is $1$, then $X(\tD_{n,m})$ is contractible -- cf.~\cite[Proposition~4.18]{secondo}. Finally, if either $n$ or $m$ is zero, and $m + n > 1$ (i.e.~if we have a source graph or a sink graph), then it is not difficult to check that $X(\tD_{n,m})$ is homotopy equivalent to the wedge of $n + m -1$ copies of the $0$-dimensional sphere. 
\end{example}

The \emph{order complex} $\Delta(P)$ of a poset $P$ is the simplicial complex whose faces are the chains of the poset. It is known that the order complex of the face poset of a complex $S$ is the barycentric subdivision of $S$. So the  order complex of the path poset $\bar{P}(\tG)=P(\tG)\setminus\{ \bar{\tK}_n\}$ (where $\bar{\tK}_n$ is the graph with $n$ vertices and no edges) is the barycentric subdivision of the multipath complex $X(\tG)$, as such, the order complex of~$\bar{P}(\tG)$ and the multipath complex $X(\tG)$ are homotopy equivalent. The reduced Euler characteristic $\widetilde{\chi}$ of the order complex of a poset is equal to the \emph{M\"obius function} of the poset, which is recursively defined as $\mu_P(u,u)=1$ and $$\mu_P(u,v)=-\sum_{u\le w<v}\mu(u,w)\ .$$
More precisely,  $\widetilde{\chi}(\Delta(P))=\mu(P):=\mu_{L(P)}(\hat{0},\hat{1})$, where $L(P)$ is obtained from $P$ by attaching a minimal element~$\hat{0}$ and a maximal element~$\hat{1}$.  
Therefore, if we consider $\hat{0}=\bar{\tK}_n$, then
\begin{equation}\label{eq:moebiuschar}
\widetilde{\chi}(X(\tG))=\widetilde{\chi}(\Delta(\bar{P}(\tG)))=\bar{\mu}(P(\tG)):=-\sum_{p\in P(G)}\mu(\bar{\tK}_n,p) .
\end{equation}
So, we can compute the reduced Euler characteristic of the multipath complex directly from the path poset. Note that throughout we refer to the reduced Euler characteristic simply as the Euler characteristic, and see \cite{Wac07} for further background on order complexes and the M\"obius function.

\begin{rem}
Denote by $\ast$ the join operation of simplicial complexes. Then, for $\tG$ and $\tH$ directed graphs, we have a homotopy equivalence 
\[
X(\tG\sqcup \tH)\simeq X(\tG)\ast X(\tH) \ ,
\]
where $\sqcup$ denotes the disjoint union of digraphs.
\end{rem}

We conclude this section with a relation between multipath complexes and matching complexes for certain families of digraphs. The latter is the simplicial complex whose simplices are collections of disjoint edges in an unoriented graph. We first need the notion of alternating orientations.
Given an orientation $o$ on an undirected graph $\tG$, we denote by $\tG_o$ the corresponding digraph. 

\begin{defn}
 An orientation $o$ on $\tG$ is called \emph{alternating} if there exists a partition~$V\sqcup W$ of $V(\tG_o)$ such that all elements of  $V$ have indegree $0$ and all elements of $W$ have outdegree $0$.
\end{defn}

Note that the existence of an alternating orientation implies that $\tG$ is a \emph{bipartite} graph (that is there exists a function $f\colon V(\tG) \to \{ 0,1\}$ that assumes distinct values on vertices which share an edge in $\tG$).
As mentioned above, alternating orientations can be used to create a bridge between multipath complexes of digraphs and the matching complexes of the underlying undirected graphs. We recall that a \emph{matching} on a graph $\tG$ is a collection of edges without common vertices. The \emph{matching complex} $M(\tG)$ is the simplical complex whose simplices are matchings on $\tG$ -- see also~\cite{torsion}. 

\begin{prop}[{\cite[Theorem~4.1]{SpriSecondo}}]
Let $\tG$ be a graph and $o$ an orientation on $\tG$. Then, we have an isomorphism of simplicial complexes \[M(\tG) \cong X(\tG_o)\] if and only if $o$ is alternating.
\end{prop}

A consequence of the proposition is that multipath complexes may have torsion -- cf.~\cite[Proposition~4.5]{SpriSecondo}.

\section{Euler characteristic of multipath complexes, and generating functions}\label{sec:generating}

The purpose of this section is to provide some examples and explicit computation of the Euler characteristic of the multipath complex of some families of  digraphs. We provide both explicit closed formulae and expressions for exponential generating functions.

 \subsection{Euler Characteristic of Complete Graph and Transitive Tournament}
 We begin by considering different orientations of the complete graph, and show that the Euler characteristic of the multipath complex of these graphs is closely linked to the number of set partitions, and their variations. First we introduce a lemma that is useful throughout.
 
Recall that the M\"obius function $\bar{\mu}(P(\tG))$ is equal to the Euler characteristic $\widetilde{\chi}(X(\tG))$ -- cf.~Equation~\ref{eq:moebiuschar}. For notational ease let $\mu(p):=\mu_{P(\tG)}(\bar{\tK}_n,p)$ when $\tG$ is clear. 
 
 \begin{lem}\label{lem:sign}
   Let $\tG$ be a digraph on $n$ vertices. For every $g\in P(\tG)$ we have $\mu(g)=(-1)^{n-k(g)}$, where $k(g)$ is the number of components of the multipath $g$.
    \end{lem}
   \begin{proof}
Let $m$ be the number of edges in $g$, then $m=n-k(g)$. This can be seen by induction since if $k(g)=n$ then the graph has no edges, and adding an edge is equivalent to connecting two components in a multipath.
    
The interval $[\bar{\tK}_n,g]$ in $P(\tG)$ is isomorphic to the Boolean lattice $\mathbb{B}(m)$ since every multipath contained in $g$ is equivalent to a subset of the edges of $g$. It is known that $\mu_{\mathbb{B}(m)}(\min,\max)=(-1)^{m}$ (e.g.~\cite[Example~1.1.1]{Wac07}), so we have:
\[\mu(g)  = \mu_{\mathbb{B}(m)}(\min,\max)=(-1)^{m}=(-1)^{n-k(g)},\]
as desired.
\end{proof}

We start by computing the Euler characteristic of $X(\tK_n)$. These complexes were studied before, and are known to be highly connected, with a bound on connectivity which depends on $n$ -- cf.~\cite[Theorem~10]{Omega}.

 \begin{thm}\label{thm:complete}
 Let $\tK_n$ be the complete digraph on $n$ vertices, that is, with a bidirectional edge between every pair of vertices. Then  \begin{equation}\label{eq:lah}\widetilde{\chi}(X(\tK_n))=\sum_{k=1}^{n}(-1)^{n-k-1}\binom{n-1}{k-1}\frac{n!}{k!}\ ,\end{equation}
 which has the exponential generating function $e^{\frac{x}{x-1}}$.
 \begin{proof}
 Let $\Pi_{n,k}$ be all ordered partitions of $[n]=\{1,\ldots,n\}$ into $k$ parts and let $\Pi_n$ be all ordered partitions of $[n]$. 
  Define a function $f\colon P(\tK_n)\rightarrow\Pi_n$, where $f(g)$ is the ordered partition where each part of $f(g)$ is the vertices in a simple path of $g$, and the order of the part is given by the position of the vertex in the simple path. It is clear that $f$ is a bijection; its inverse is given by converting every part of a partition into a simple path, which makes a valid multipath as all simple paths are possible in $\tK_n$. 
  
  By Lemma~\ref{lem:sign} we know that $\mu(g)=(-1)^{n-k}$ for all $f(g)\in\Pi_{n,k}$ and it is known that $|\Pi_{n,k}|=\binom{n-1}{k-1}\frac{n!}{k!}$ -- these are the Lah numbers, see \cite{Pet07} or OEIS sequence A105278 \cite{OEIS}. So we get
  $$\widetilde{\chi}(X(\tK_n))= \bar{\mu}(P(\tK_n))=-\sum_{k=1}^{n}(-1)^{n-k}|\Pi_{n,k}|=\sum_{k=1}^{n}(-1)^{n-k-1}\binom{n-1}{k-1}\frac{n!}{k!}\ .$$
  
  If we replace $(-1)^{n-k-1}$ with $(-1)^{k-1}$ in Equation~\ref{eq:lah} we get OEIS Sequence A066668, which has exponential generating function $e^{\frac{x}{x+1}}$. 
  Since this corresponds to the sequence $(-1)^n\tilde{\chi}(\tK_n)$, we obtain the desired exponential generating function.
 \end{proof}
 \end{thm}
 
 We believe that the multipath complex of the complete graph $\tK_n$ has the largest Euler characteristic of any graph with $n$ vertex. As such we make the following conjecture, which has been verified computationally for $n<8$ using~\cite{pathposet}.
 \begin{conj}
 Let $\tG$ be any digraph on $n$ vertices, then
 $\widetilde{\chi}(X(\tK_n))\ge \widetilde{\chi}(X(\tG))$.
 \end{conj}

The \emph{transitive tournament} on $n$ vertices is the unique (up to isomorphism) orientation of the complete undirected graph with no directed cycles. This is equivalent to taking the complete undirected graph and orient all edges from smaller vertex index to larger. We now show that the Euler characteristic of the multipath complex of transitive tournaments is given by a variation of the complementary Bell numbers, that is, the alternating sum of the Stirling numbers.

 \begin{thm}\label{thm:transitive}
 Let $\tT_n$ be the transitive tournament on $n$ vertices. Then  \begin{equation}\label{eq:stirl}\widetilde{\chi}(X(\tT_n))=\sum_{k=1}^{n}(-1)^{n-k-1}S(n,k)\ ,\end{equation}
 where $S(n,k)$ are the Stirling numbers of the second kind and sequence given by Equation~\eqref{eq:stirl} has the exponential generating function $-e^{1-e^{-x}}$.
 \begin{proof}
 Let $\Pi_{n,k}$ be all partitions of $[n]$ into $k$ parts and let $\Pi_n$ be all partitions of $[n]$. Proceeding as in the previous proof, define a function $f\colon P(\tT_n)\rightarrow\Pi_n$, where $f(g)$ is the partition where each part of $f(g)$ is the vertices in a simple path of $g$. It is clear that $f$ is a bijection as the inverse is given by converting every part of a partition into a simple path, and in a transitive tournament there is a unique way to make a simple path from a set of vertices.
  
  By Lemma~\ref{lem:sign} we know that $\mu(g)=(-1)^{n-k}$ for all $f(g)\in\Pi_{n,k}$. Therefore,
  $$\widetilde{\chi}(X(\tT_n))=\bar{\mu}(P(\tK_n))=-\sum_{k=1}^{n}(-1)^{n-k}|\Pi_{n,k}|=\sum_{k=1}^{n}(-1)^{n-k-1}S(n,k)\,$$
  since the number of partitions is exactly the Stirling numbers of the second kind.
  
  The alternating sum of the Stirling numbers are known as the complementary Bell numbers, sequence A000587 in the OEIS \cite{OEIS}, for which the exponential generating function is $e^{1-e^{x}}$. However, we have $(-1)^{n-k-1}$ instead of~$(-1)^k$ so we must negate the even term in the sequence obtaining the exponential generating $-e^{1-e^{-x}}$.
 \end{proof}
 \end{thm}

 Next we consider what happens if we reverse a single edge of the transitive tournament, in particular the edge $(1,n)$.
 
  \begin{thm}\label{thm:transitive2}
  Let $\tR_n$ be the graph obtained from the transitive tournament $\tT_n$ by reversing the orientation of the edge~$(1,n)$.
 For $n\ge3$ we get:
\begin{equation}\label{eq:aTT}\widetilde{\chi}(X(\tR_n))=\sum_{k=1}^{n-2}(-1)^{n-k-1}k S(n-2,k)\ ,\end{equation}
 where $S(n,k)$ are the Stirling numbers of the second kind, and $(1-e^{-x})e^{1-e^{-x}}$ is the exponential generating function for the sequence $a_n=\widetilde{\chi}(X(\tR_{n+2}))$.
 \begin{proof}
  Partition the elements of $P(\tR_n)$ into three parts $A$, $B$ and $C$, where
  \begin{enumerate}
      \item $A$ is the set of all multipaths which contain the edge $(n,1)$;
      \item $B$ is the set of multipaths which do not contain the edge $(n,1)$, but $(n,1)$ can be added to make a multipath;
      \item $C$ is the set of multipaths which do not contain the edge $(n,1)$, and $(n,1)$ cannot be added to make a multipath.
  \end{enumerate}
  Define a function $\phi\colon A\rightarrow B$ where $\phi(x)$ is the multipath obtained by removing the edge $(n,1)$
  from $x$, for all $x\in A$. Then $\phi$ has a clear inverse, which is to add in the edge $(n,1)$, so this is a
  bijection. Moreover, by Lemma~\ref{lem:sign} we get that $\mu(\phi(x))=-\mu(x)$. 
  Therefore $\sum_{x\in A}\mu(x)+\sum_{x\in B}\mu(x)=0$ so
  $$\bar{\mu}(P(\tR_n))=-\sum_{x\in P(\tR_n)}\mu(x)=-\left(\sum_{x\in A}\mu(x)+\sum_{x\in B}\mu(x)+\sum_{x\in C}\mu(x)\right)=-\sum_{x\in C}\mu(x)\ .$$
  
  Now consider the elements of $C$. If adding the edge $(n,1)$ is forbidden it must either make a cycle
  or cause a vertex to have in or out degree greater than 1. It is not possible for $n$ to have out-degree greater
  than 1, since in $\tR_n$ there is only one outgoing edge from $n$, which is $(n,1)$, similarly $1$ cannot have in-degree greater than
  1. So every element of $c\in C$ must forbid $(n,1)$ because adding it would make a cycle, which means $c$ must
  contain a path from $1$ to $n$. 
  
  Therefore, every multipath of $C$ can be constructed by taking a multipath $g$ on $[2,n-1]:=\{2,\ldots,n-1\}$,
  selecting one of the simple paths of $g$, connecting $1$ to the
  start of the simple path, and connecting the end of the simple path to $n$.
  Note that graph induced on $\tR_n$ by vertices $[2,n-1]$ is a transitive tournament, and by the proof of Theorem~\ref{thm:transitive} there are $S(n-2,k)$ multipaths on $[2,n-1]$ with $k$ components. From each of these we can construct $k$ elements of $C$, so we get $kS(n-2,k)$ multipaths in $C$ with $k$ components, and by Lemma~\ref{lem:sign} each such element $x$ has $\mu(c)=(-1)^{n-k}$, so we get
  $$\widetilde{\chi}(X(\tR_n))=\bar{\mu}(P(\tR_n))=-\sum_{x\in C}\mu(x)=-\sum_{k=1}^{n-2}(-1)^{n-k}kS(n-2,k).$$
  
 The OEIS sequence A101851 \cite{OEIS} is given by $a_n=\sum_{k=1}^{n}(-1)^{n-k}kS(n,k)$ and has exponential generating function ${(e^{-x}-1)e^{1-e^{-x}}}$. Considering the sequence $-a_{n}$, instead of $a_n$, gives the required function.
 \end{proof}
 \end{thm}
 
 \subsection{Generating function of bipartite digraphs}\label{sec:genfunbip}

Consider the complete bipartite digraph $\tK_{n,m}$, that is the digraph with vertices~$v_1,\dots,v_n$, $w_1,\dots,w_m$, and edges  $\{ (v_i, w_j )\}_{i,j}$. We concisely write~$\widetilde{\chi}_{n,m}$ for $\widetilde{\chi}(X(\tK_{n,m}))$.
Let $\mathcal{F}(x,y)$ be the \emph{mixed} generating function for $\widetilde{\chi}_{n,m}$ defined by the formula
\[\mathcal{F}(x,y)=\sum_{n,m \geq 0} \widetilde{\chi}_{n,m}\frac{y^n \,x^m}{m!}\ ;\]
we show that it admits a simple expression in terms of elementary functions.
The techniques employed here as well, as more general approaches, are extensively described in \cite{genfun}. We will need the following;

\begin{rem}\label{rem:prodser}
Let $a_i$ and $b_i$ be two sequences of integers, and consider their generating functions $A(t)=\sum_{i\geq 0} a_i t^i$ and~$B(z)=\sum_{i\geq 0} b_i \frac{z^i}{i!}$. Then, the coefficient of $t^m$ in the series $A(t)B(z)$ is
$\sum_{i=0}^m \frac{a_i \, b_{m-i}}{(m-i)!} t^i z^{m-i}$.
\end{rem}

Now, we are ready to prove the following theorem. Note that Equation~\eqref{eq:charDnmplain} already appeared in~\cite[Section~2]{MatchingChess}.

\begin{thm}\label{thm:Knm}
The Euler characteristic of $X(K_{n,m})$ is given by the closed formula
\begin{equation}\label{eq:charDnmplain}
\widetilde{\chi}_{n,m}=\sum_{k=0}(-1)^{k+1}\dbinom{m}{k}\dbinom{n}{k}k!\ ,\hskip 25pt \forall n,m>0\ ,
\end{equation}
satisfies the recurrence relation
\begin{equation}\label{eq:charDnm}
 \widetilde{\chi}_{n,m}=\widetilde{\chi}_{n-1,m}-m \widetilde{\chi}_{n-1,m-1}\ ,
\end{equation}
and the mixed generating function for $\widetilde{\chi}_{n,m}$ is
$$\mathcal{F}(x,y)= \frac{e^x}{1-y+xy}\ .$$
\begin{proof}
We begin with the closed formula.
Every multipath of length $k$ in $P(\tK_{n,m})$ is a matching of some elements of $v_1,\dots,v_n$ to some elements of $w_1,\dots,w_m$. So every multipath $m$ of length $k$ can be constructed by first choosing which elements of $w_1,\dots,w_m$ are matched to something, giving $\binom{m}{k}$ choices, and then choosing which elements of $v_1,\dots,v_n$ they are matched to, giving $\frac{n!}{(n-k)!}$ choices. And by Lemma~\ref{lem:sign} we know that $\mu(m)=(-1)^k$. Combining the above, summing over $k$ and negating gives the closed formula for the M\"obius function $\bar{\mu}(P(\tK_{n,m}))$, and thus $\widetilde{\chi}_{n,m}$.

Next we give a recurrence relation for $\widetilde{\chi}_{n,m}$. Partition $P(\tK_{n,m})$ into parts $P_0,\ldots,P_m$, where $P_0$ contains all multipaths that do not have an edge with source $v_1$, and $P_j$ contains all multipaths which contain the edge $(v_1,w_j)$, for all $j>0$. By the definition of the M\"obius function and since we have a partition we know that \begin{equation}\label{eq:mupart}\bar{\mu}(P(\tK_{n,m}))=-\sum_{i=0,\ldots,m}\sum_{p\in P_i}\mu(p)\ .\end{equation}

Since $v_0$ is an isolated vertex in all multipaths of $P_0$, we get that $P_0$ is isomorphic to the poset $P(\tK_{n-1,m})$.  Moreover, each of the $P_j$'s is isomorphic to $P(\tK_{n-1,m-1})$, where the isomorphism $f_j$ is the map which removes the vertices $v_1$ and $w_j$, and the edge $(v_1,w_j)$. So 
\begin{equation}\label{eq:partsum}-\sum_{p\in P_0}\mu(p)=\bar{\mu}(P(\tK_{n-1,m}))\hskip 30pt\text{ and }\hskip 30pt-\sum_{p\in P_j}\mu(p)=-\bar{\mu}(P(\tK_{n-1,m-1}))\ ,\end{equation} where the negation of $\bar{\mu}(P(\tK_{n-1,m-1}))$ is caused by $f_j$ removing an edge hence $\mu(p)=-\mu(f_j(p))$.
Combining \eqref{eq:mupart} and \eqref{eq:partsum}, and replacing $\bar{\mu}$ with the Euler characteristic gives the recurrence relation \eqref{eq:charDnm}.

Finally, we compute the generating function. Consider the generating function for the Euler characteristic for a fixed $m$, that is the function
\[
F_m(y):= \sum_{j\geq 0} \widetilde{\chi}_{j,m}y^j \ .
\]
It follows from the definitions that $\widetilde{\chi}_{0,m}=\widetilde{\chi}_{n,0}=1$, and thus $F_0(y)=\frac{1}{1-y}$.
By multiplying the recursive relation in Equation~\eqref{eq:charDnm} by $y^{n-1}$, summing up over $n>0$, and rearranging the terms, one obtains that  
$(1-y)\, F_m(y)=-m y \;F_{m-1}(y)+1$.
Consequently, it follows: 
\begin{align*}\label{eq:GENcharDnm}F_m(y)&=\frac{-my}{(1-y)}F_{m-1}(y) + \frac{1}{1-y} = \\&= (-1)^m \frac{m!\, y^m}{(1-y)^m}F_0(y)+\sum_{i=0}^{m-1} \frac{ m!}{(m-i)!} \frac{(-1)^i \, y^i}{(1-y)^{i+1}} = \\&=\sum_{i=0}^{m} \frac{ m!}{(m-i)!}\frac{(-1)^i\,y^i}{(1-y)^{i+1}} \ . \end{align*}
We can now find an explicit formula for the exponential generating function of the $F_m(y)$, which means: 
\begin{align*}\mathcal{F}(x,y) &=\sum_{m \geq 0} F_m(y)\frac{x^m}{m!} = 
\frac{1}{(1-y)} \sum_{m \geq 0}\left[\sum_{i=0}^{m} \frac{ 1}{(m-i)!}\frac{(-1)^i\,y^i}{(1-y)^{i}} \right]x^m \ . \end{align*}
In virtue of Remark~\ref{rem:prodser}, taking $b_i=a_i=1$, and setting $t=\frac{-xy}{(1-y)}$, and $z=x$, one obtains
\begin{align*}
\frac{e^x}{1-\frac{-xy}{(1-y)}}=A(t)_{t=\frac{-xy}{(1-y)}}\;B(z)_{z=x} = \sum_{m \geq 0} \left[\sum_{i=0}^{m} \frac{1}{(m-i)!}\frac{(-1)^i\,y^i}{(1-y)^{i}}\right]x^m \ ;
\end{align*}
consequently, we get
\begin{equation*}\label{eq:GenFunBipartite}
\mathcal{F}(x,y) =\frac{1}{(1-y)} \sum_{m \geq 0}\left[\sum_{i=0}^{m} \frac{ 1}{(m-i)!}\frac{(-1)^i\,y^i}{(1-y)^{i}} \right]x^m =\frac{e^x}{1-y+xy} \ , 
\end{equation*}
which provides the desired formula.
\end{proof}
\end{thm}

Note that the generating function $\mathcal{F}(x,y)$ is a \emph{mixed} generating function for the Euler characteristic, ordinary with respect to $n$ and exponential with respect to $m$. This implies  that the symmetric role of $n$ and $m$ is not reflected on~$\mathcal{F}(x,y)$.
 We remark that reversing the orientation of all edges does not change the path poset, hence we have the equality $\widetilde{\chi}_{n,m}=\widetilde{\chi}_{m,n}$. As a consequence, the generating function $F_m(y)$ coincides with the generating function
\[ G_{n}(x) = \sum_{i \geq 0} \widetilde{\chi}_{n,i} x^{i} \ ,\]
obtained by considering bipartite complete graphs with a fixed number of \emph{sources}.
The generating function $\mathcal{F}(x,y)$ is in fact a (mixed) generating function of the Euler characteristic of the \emph{chessboard complex}, i.e.~the matching complex of (the  underlying unoriented graph of) $\tK_{n,m}$.

\begin{rem}
The number of multipaths of $\tK_{n,m}$ is given by OEIS sequence A088699 \cite{OEIS}, and has generating function $$\mathcal{F}'(x,y)= \frac{e^x}{1-y-xy}.$$
Note the difference in sign for $xy$, which causes the alternating sum of the multipaths by length, given in Theorem~\ref{thm:Knm}.
\end{rem}

 \section{Dynamical regions and computations }\label{sec:modules}
 
 In this section we introduce a decomposition of directed graphs into subgraphs called dynamical regions. We use minimal decompositions in to dynamical regions to simplify the digraph complexity, hence to compute the homotopy type of multipath complexes. We provide the full computations for the families of linear graphs, polygons and small grids.

  \subsection{Dynamical regions and modules}
 Let $\tG$ be a digraph, and let $\tG'\leq \tG$ be a subgraph. We will use the following terminology. The \emph{complement} $\Comp{\tG'}{\tG}$ of  $\tG'$ in $\tG$ is the subgraph of $\tG$ spanned by the edges in $E(\tG)\setminus E(\tG')$. The \emph{boundary} $\partial_{\tG} \tG'$ of $\tG'$ in $\tG$, or simply $\partial \tG'$ when clear from the context, is defined as $\partial_{\tG} \tG'=V(\tG')\cap V(\Comp{\tG'}{\tG})$, see~Figure~\ref{fig:boundary and complement} for an example. 

\begin{defn}
Let $\tG$ be a connected digraph with at least one edge. A vertex $v\in V(\tG)$ is called \emph{stable} if either the indegree or the outdegree of $v$ is zero, and \emph{unstable} otherwise. 
\end{defn}

The following is the main definition of the section;

\begin{defn}\label{def:dynamicalregion}
Let $\tG$ be a digraph.
A \emph{dynamical region} in $\tG$ is a connected full subgraph $\tR \leq \tG$, with at least one edge, such that: 
\begin{enumerate}[label = (\alph*)]
    \item\label{item:dandelion} all vertices in the boundary of $\tR$ are unstable in $\tG$, but stable in both $\tR$ and $\Comp{\tR}{\tG}$;
    \item\label{item:cicle} the edges of $\tR$ do not belong to any oriented cycle in $\tG$ which is not contained in $\tR$.
\end{enumerate}
A dynamical region is called \emph{stable} (resp.~\emph{unstable}) if all its non-boundary vertices are stable (resp.~unstable, and at least one vertex is unstable).
\end{defn}

\begin{rem}\label{rem:intersectionregions}
The non-empty intersection of two dynamical regions, say $\tR $ and $\tS$, still satisfies \ref{item:dandelion} and \ref{item:cicle}. In particular, each connected component of $\tR \cap \tS$ is still a dynamical region.
\end{rem}

Observe that item \ref{item:dandelion} is equivalent to asking that, for each vertex $v\in \partial\tR$, all edges incident to $v$ belonging $E(\tG)\setminus E(\tR)$ have opposite orientation with respect to the edges in $\tR$ incident to $v$. We will also say that the vertices in the boundary are \emph{coherent dandelions}.

\begin{defn}
A \emph{dynamical module}, shortly a \emph{module}, $\tM$ of a digraph $\tG$ is a minimal dynamical region.
\end{defn}

 For a digraph $\tG$, its associated cone is the digraph $\mathrm{Cone} (\tG)$ with vertices $V(\tG) \cup \{ v_0 \}$ and edges~$E(\tG) \cup \{ (v, v_0) \mid v\in V(\tG)\}$. Coning is a good way to produce modules which are not stable dynamical regions -- cf.~Example~\ref{ex:alternating} -- e.g.~transitive tournaments.
%

\begin{example}
A dynamical region which is a dandelion subgraph is never a module, unless it is of type $\tD_{n,0}$ (or $\tD_{0,n}$). In general $\tD_{m,n}$ splits as the union of two dynamical modules: one copy of $D_{m,0}$ and a copy of $D_{0,n}$.
\end{example}

  \begin{figure}[h]
    \centering
    \begin{tikzpicture}[baseline=(current bounding box.center)]
		\tikzstyle{point}=[circle,thick,draw=black,fill=black,inner sep=0pt,minimum width=2pt,minimum height=2pt]
		\tikzstyle{arc}=[shorten >= 8pt,shorten <= 8pt,->, thick]
		
		\node[above] (v0) at (0,0) {$v_0$};
		\draw[fill] (0,0)  circle (.05);
		\node[above] (v1) at (1.5,0) {$v_1$};
		\draw[fill] (1.5,0)  circle (.05);
		\node[above] (v2) at (3,0) {$v_2$};
		\draw[fill] (3,0)  circle (.05);
		\node[above] (v4) at (4.5,0) {$v_3$};
		\draw[fill] (4.5,0)  circle (.05);
		\node[above] (v5) at (6,0) {$v_4$};
		\draw[fill] (6,0)  circle (.05);
		\node[above] (v6) at (7.5,0) {$v_5$};
		\draw[fill] (7.5,0)  circle (.05);
		 \node[above]  at (9,0) {$v_{n-1}$};
		 \node (v7) at (9,0) {};
		\draw[fill] (9,0)  circle (.05);
		
		\node   at (8.25,0) {$\dots$};
		 \node (v8) at (10.5,0) {};
		 \node[above]  at (10.5,0) {$v_n$};
		\draw[fill] (10.5,0)  circle (.05);
		
		\draw[thick, bunired, -latex] (0.15,0) -- (1.35,0);
		\draw[thick, bunired, -latex] (2.75,0) -- (1.65,0);
		\draw[thick, bunired, latex-] (4.35,0) -- (3.15,0);
		\draw[thick, bunired, latex-] (4.65,0) -- (5.85,0);
		\draw[thick, bunired, latex-] (7.35,0) -- (6.15,0);
	    \draw[thick, bunired, ] (v7) -- (v8);

	\end{tikzpicture}
	\caption{The alternating graph $\tA_n$ on $n+1$ vertices. The edge between $v_{n -1} $ and $v_{n}$ can be oriented either way depending on the parity of $n$. }
    \label{fig:alternating}
\end{figure}
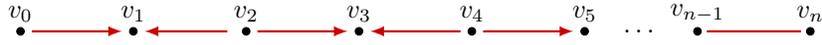

 \begin{example}\label{ex:alternating}
 An alternating graph $\tA_n$ -- cf.~Figure~\ref{fig:alternating} -- is a module. More generally, a stable dynamical region is a module (since each vertex has either outdegree or indegree $0$).
 \end{example}

 \begin{example}
 Consider the digraph $\tG$ in Figure~\ref{fig:boundary and complement}. The subgraph in blue is not a dynamical region of $\tG$, as it is not connected; its leftmost connected component is a module, as it is connected, no edges are contained in any oriented cycles of $\tG$ and the $1$-neighbourhoods of vertices in its boundaries are coherent dandelions. The rightmost connected component instead is not a module, because its only edge is contained in a directed cycle of $\tG$.
  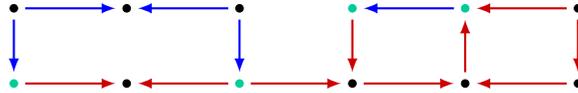
\begin{figure}[h]
     \centering
     \begin{tikzpicture}[scale=1.0][baseline=(current bounding box.center)]
					\tikzstyle{point}=[circle,thick,draw=black,fill=black,inner sep=0pt,minimum width=2pt,minimum height=2pt]
					\tikzstyle{arc}=[shorten >= 8pt,shorten <= 8pt,->, thick]
					\node[above] (v0) at (0,1) {};
					\draw[fill] (0,1)  circle (.05);
					\node[above] (w0) at (4.5,1) {};
					\draw[fill, caribbeangreen] (4.5,1)  circle (.05);
					\node[above] (w1) at (6.0,1) {};
					\draw[fill, caribbeangreen] (6.0,1)  circle (.05);
					\node[above] (w2) at (7.5,1) {};
					\draw[fill] (7.5,1)  circle (.05);
					\node[above] (v1) at (1.5,1) {};
					\draw[fill] (1.5,1)  circle (.05);
					\node[above] (v2) at (3,1) {};
					\draw[fill] (3,1)  circle (.05);
					\draw[thick, blue, -latex] (0.15,1) -- (1.35,1);
					\draw[thick, blue, -latex] (2.85,1) -- (1.65,1);
					\draw[thick, blue, -latex] (5.85,1) -- (4.65,1);
					\draw[thick, bunired, latex-]  (6.15,1) -- (7.35,1) ;
					\node[above] (v0) at (0,0) {};
					\draw[fill, caribbeangreen] (0,0)  circle (.05);
					\node[above] (w0) at (4.5,0) {};
					\draw[fill] (4.5,0)  circle (.05);
					\node[above] (w1) at (6.0,0) {};
					\draw[fill] (6.0,0)  circle (.05);
					\node[above] (w2) at (7.5,0) {};
					\draw[fill] (7.5,0)  circle (.05);
					\node[above] (v1) at (1.5,0) {};
					\draw[fill] (1.5,0)  circle (.05);
					\node[above] (v2) at (3,0) {};
					\draw[fill, caribbeangreen] (3,0)  circle (.05);
					\draw[thick, bunired, -latex] (0.15,0) -- (1.35,0);
					\draw[thick, bunired, -latex] (2.85,0) -- (1.65,0);
					\draw[thick, bunired, -latex]  (3.15,0) -- (4.35,0) ;
					\draw[thick, bunired, latex-] (5.85,0) -- (4.65,0);
					\draw[thick, bunired, latex-]  (6.15,0) -- (7.35,0) ;
					\draw[thick, blue, -latex] (0,0.85) -- (0,0.15);
					\draw[thick, blue, -latex] (3,0.85) -- (3,0.15);
					\draw[thick, bunired, -latex] (4.5,0.85) -- (4.5,0.15);
					\draw[thick, bunired, latex-] (6,0.85) -- (6,0.15);
					\draw[thick, bunired, -latex] (7.5,0.85) -- (7.5,0.15);
			\end{tikzpicture}
     \caption{A graph $\tG$, a subgraph $\tH$ (in blue) and its complement (in red). The boundary of $\tH$ in $\tG$ is represented in green.}
     \label{fig:boundary and complement}
 \end{figure}
 \end{example}

The following is straightforward from the definitions:

 \begin{lem}\label{lem:alternatingmultipath}
The multipath complex of a stable dynamical region $\tR$ in $\tG$ is the matching complex of the underlying unoriented graph of $\tR$.
\end{lem}

For a digraph $\tG$, a decomposition in to dynamical regions allows us to decompose the multipath complexes in to smaller complexes. In fact, we have the following result:

 \begin{prop}\label{prop:regions}
 If  $\tR\leq \tG$ is a dynamical region, and we set $\tS\coloneqq \Comp{\tR}{\tG}$, then we have the homotopy equivalence
 \[
 X(\tG) \simeq X(\tR) \ast X(\tS)
 \]
 between the associated multipath complexes.
\end{prop}

 \begin{proof}
 Observe that, if $\tR\leq \tG$ is a dynamical region, then the vertices in the boundary of $\tR$ are coherent dandelions. Let~$\tH$ be a multipath of $\tG$; then $\tH\cap \tR$ and $\tH \cap \tS$ are multipaths in $\tR$ and $\tS$, respectively. 
 Vice versa, if $\tH$ and $\tH'$ are multipaths of $\tR$ and $\tS$, respectively, then $\tH\cup \tH'$ is a multipath of $\tG$ as no edges of $\tR$ are contained in any oriented cycle of $\tG$ and the edges in the boundary compose. As a consequence, the path poset of $\tG$ is isomorphic to the path poset of the disjoint union of $\tR$ and $ \tS$. 
 
 The multipath complex of $\tG$ can now be identified with the multipath complex of the disjoint union $\tR\sqcup \tS$.
  To conclude, observe that the multipath complex of the disjoint union of two directed graphs is homotopic to the join of the multipath complexes -- compare~\cite[Definition 2.16]{Kozlov} and \cite[Remark~3.2]{secondo}.
 \end{proof}

\begin{lem}
For each edge $e\in E(\tG)$ there exists a unique dynamical module of $\tG$ containing it. 
\end{lem}

\begin{proof}
The statement follows from Remark~\ref{rem:intersectionregions};  taking the intersection of all the dynamical regions in $\tG$ containing the edge~$e$. This satisfies \ref{item:dandelion} and \ref{item:cicle}
in Definition~\ref{def:dynamicalregion}, and it is connected.
It is also unique by construction, which concludes the proof. 
\end{proof}

Observe that the construction of the (unique) dynamical module containing a subset~$S$ of edges of $\tG$ can be performed iteratively. In fact, this is achieved by repetitively applying the following steps:
\begin{enumerate}
    \item for each edge $e$ in $S$, add to $S$ all the edges $ e'$ of $\tG$ with target $t(e')=t(e)$ or source $s(e')=s(e)$;
    \item for each edge $e$ in $S$ contained in a coherent cycle $\Gamma$ of $\tG$, add to $S$ all the edges $ e''$ with $e''\in \Gamma$.
\end{enumerate}
As a corollary, we get:
  
 \begin{thm}\label{thm:module} 
We have a unique (up to re-ordering) decomposition of $\tG$ in to dynamical modules $\tM_1,\dots,\tM_{k}$, and 
 \[X(\tG) \simeq X(\tM_1) \ast \cdots \ast X(\tM_k) \ .\]
 Furthermore, this decomposition can be found algorithmically.
 \end{thm}
 
 \begin{proof}
 Fix an edge $e$ of $\tG$. This is contained in a unique module $\tM_{e}$, and $X(\tG) \simeq X(\tM_e) \ast X(\Comp{\tM_{e}}{\tG})$ by Proposition~\ref{prop:regions}. Now, we can proceed iteratively, by considering $\Comp{\tM_{e}}{\tG}$ \emph{en lieu} of $\tG$.
 This provides the desired decomposition, and since this decomposition is given by the unique modules containing each edge in $\tG$, uniqueness follows. 
 \end{proof}
 
In particular, we have that if one of the modules in the decomposition of $\tG$ has contractible multipath complex, then~$X(\tG)$ is contractible (and hence has trivial reduced cohomology).

\subsection{Multipath complexes of  polygonal graphs}

In this section we apply Theorem~\ref{thm:module} to the computation of the homotopy type of multipath complexes of linear and polygonal graphs; here, by polygonal graph, we mean any oriented (i.e.~no bi-directional edges) graph whose underlying undirected graph is a cycle. We first need a definition.

 \begin{defn}
  The \emph{size} of a dynamical region is the number of its non-boundary vertices.
 \end{defn}

\begin{lem}
Let $\tP$ be a polygonal graph with at least a stable vertex. If $\tP$ has an unstable region of size at least two, then $X(\tP)$ is contractible.
\end{lem}
\begin{proof}
The presence of an unstable region $S$ with at least two non-boundary vertices implies, since $\tP$ is not coherently oriented, that we can take as a module any edge between two non-boundary vertices in $S$. This implies that $X(\tP)$ is homotopy equivalent to a cone, hence contractible.
\end{proof}

\begin{prop}
Let $\tP$ be a polygonal graph with no unstable vertices. Then, the number $n$ of vertices in $\tP$ is even, and 
\[ X(\tP) \simeq \begin{cases} S^{k-1} \vee S^{k-1} & \text{if } n = 3k\ ,\\  S^{k-1} & \text{if }n = 3k + 1\ ,\\ S^{k} & \text{if }n = 3k + 2 \ .\\\end{cases} \]
In particular, the associated multipath complex is always homotopy equivalent to a wedge of spheres.
\end{prop}
\begin{proof}
If there are no unstable vertices, then the orientation on $\tP$ is alternating, which implies that the number of vertices is even.
Therefore, the multipath complex coincides with the matching complex, see~Lemma~\ref{lem:alternatingmultipath}.
Kozlov in \cite[Proposition~5.2]{Kozlov99} computes the matching complex of the cycle~$\mathscr{C}_n$ with $n$ vertices, borrowing his notation. This turns out to be either a sphere or the wedge of two spheres, whose dimension depends only on the number of vertices modulo~$3$. The assertion now follows directly from Kozlov's statement.
\end{proof}

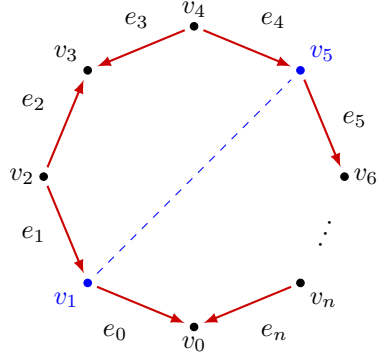
\begin{figure}[h]
\centering
\newdimen\R
\R=2.0cm
\begin{tikzpicture}
\draw[xshift=5.0\R, fill] (270:\R) circle(.05)  node[below] {$v_0$};
\draw[xshift=5.0\R,fill, blue] (225:\R) circle(.05)  node[below left]   {$v_1$};
\draw[xshift=5.0\R,fill] (180:\R) circle(.05)  node[left] {$v_2$};
\draw[xshift=5.0\R,fill] (135:\R) circle(.05)  node[above left] {$v_3$};
\draw[xshift=5.0\R, fill] (90:\R) circle(.05)  node[above] {$v_4$};
\draw[xshift=5.0\R,fill, blue] (45:\R) circle(.05)  node[above right] {$v_5$};
\draw[xshift=5.0\R,fill] (0:\R) circle(.05)  node[right] {$v_6$};
\draw[xshift=5.0\R,fill] (315:\R) circle(.05)  node[below right] {$v_{n}$};

\node[xshift=5.0\R] (v0) at (270:\R) { };
\node[xshift=5.0\R] (v1) at (225:\R) { };
\node[xshift=5.0\R] (v2) at (180:\R) { };
\node[xshift=5.0\R] (v3) at (135:\R) { };
\node[xshift=5.0\R] (v4) at (90:\R) { };
\node[xshift=5.0\R] (v5) at (45:\R) { };
\node[xshift=5.0\R] (v6) at (0:\R) { };
\node[xshift=5.0\R] (vn) at (315:\R) { };

\draw[thick, bunired, latex-] (v0)--(v1);
\draw[thick, bunired, latex-] (v1)--(v2);
\draw[thick, bunired, -latex] (v2)--(v3);
\draw[thick, bunired, latex-] (v3)--(v4);
\draw[thick, bunired, -latex] (v4)--(v5);
\draw[thick, bunired, -latex] (v5)--(v6);
\draw[thick, bunired, -latex] (vn)--(v0);

\draw[dashed, blue] (v1)--(v5);

\draw[xshift=5.0\R, fill] (292.5:\R) node[below right] {$e_{n}$};
\draw[xshift=5.0\R,fill] (247.5:\R) node[below left] {$e_0$};
\draw[xshift=5.0\R,fill] (202.5:\R)   node[left] {$e_1$};
\draw[xshift=5.0\R,fill] (157.5:\R)  node[above left] {$e_2$};
\draw[xshift=5.0\R, fill] (112.5:\R)   node[above] {$e_3$};
\draw[xshift=5.0\R,fill] (67.5:\R) node[above right] {$e_4$};
\draw[xshift=5.0\R,fill] (22.5:\R) node[right] {$e_5$};
\draw[xshift=4.95\R,fill] (337.5:\R)  node {$\cdot$} ;
\draw[xshift=4.95\R,fill] (333:\R)  node {$\cdot$} ;
\draw[xshift=4.95\R,fill] (342:\R)  node {$\cdot$} ;
\end{tikzpicture}
\caption{A polygonal graph on $n$ edges with (at least) two vertices that are neither sources or sinks (in blue). The dashed line shows the separation between the two modules.} 
\label{fig:polydash}
\end{figure}

By the previous results, we might assume that the considered polygonal graph~$\tP$ has unstable regions of size at most one, and at least an unstable region. The unstable vertices can be used to split~$\tP$ into modules which are alternating linear graphs -- cf.~Figure~\ref{fig:polydash}. More precisely, we have the following result;

\begin{prop}\label{lem:polygonalgraphs}
Let \tP~be a polygonal graph with at least one stable vertex, and no unstable regions of size greater than one. Denote by $\ell_1,...,\ell_k$ the size of the stable regions, then
\[ X(\tP) \simeq X(\tA_{\ell_1 +2}) \ast \cdots \ast X(\tA_{\ell_k +2}) \ . \]
In particular, $X(\tP)$ is contractible if, for some $i$, $\ell_i = 3i - 1 $  and
\[ X(\tP) \simeq S^{\left\lceil \frac{\ell_1 - 1}{3} \right\rceil} \ast \cdots \ast  S^{\left\lceil \frac{\ell_k - 1}{3} \right\rceil} \ ,  \]
otherwise.\end{prop}
\begin{proof}
The unstable vertices are the boundary of certain modules. These modules, which correspond to stable regions, are alternating linear graphs with as many vertices as the size of the corresponding stable region, plus two (given by the unstable vertices bounding the region). By Lemma~\ref{lem:alternatingmultipath} and Kozlov's computations \cite[Proposition~4.6]{Kozlov99}, the multipath complex associated to an alternating graph $\tA_k$ with $k$ vertices is contractible if and only if $k = 3s + 1$, while it is homotopy equivalent to $S^{\lceil (k-1)/3 \rceil}$ otherwise. The statement follows.
\end{proof}
 
We conclude by observing that the same reasoning used to determine $X(\tP)$ works almost verbatim for linear graphs. In particular, one can obtain a precise description of the homotopy type of $X(\tL)$ for each linear graph $\tL$, which can be used to recover \cite[Theorem~1.1]{secondo}.

\subsection{Multipath complexes of small grids} 
 
Aim of this subsection is to compute the homotopy type of multipath complexes of small grids of type $\tL \times \tI_m$, where $\tL$ is a linear graph and $\tI_m$ a coherent linear graph. 
By \cite[Example~4.20]{primo}, the multipath cohomology groups of coherent linear graphs are trivial. We compute here the homotopy type of $X(\tI_n\times \tI_m)$.

 \begin{prop}\label{prop:coherent products}
 Let $n, m$ be non-negative integers, then
 \[
 X(\tI_n\times \tI_m)\simeq 
 \begin{cases}
\ * & \text{ if } \ n,m\neq 1\\
\ S^n & \text{ if } \ m=1\\
S^m & \text{ if } \ n=1
 \end{cases}
 \]
 \end{prop}
 
 \begin{proof}
 The case $n$ or $m$ equal to $0$ is covered in \cite[Example~4.20]{primo}. Assume that $m=1$, the case $n=1$ being analogous. The decomposition in to dynamical modules of $\tI_n\times \tI_1$ is shown in Figure~\ref{fig:gridI_nxI1}.
 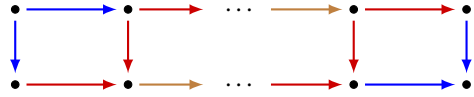
\begin{figure}[H]
     \centering
     \begin{tikzpicture}[scale=1.0][baseline=(current bounding box.center)]
					\tikzstyle{point}=[circle,thick,draw=black,fill=black,inner sep=0pt,minimum width=2pt,minimum height=2pt]
					\tikzstyle{arc}=[shorten >= 8pt,shorten <= 8pt,->, thick]
					\node[above] (v0) at (0,0) {};
					\draw[fill] (0,0)  circle (.05);
					\node[above] (v1) at (1.5,0) {};
					\draw[fill] (1.5,0)  circle (.05);
					\node[] at (3,0) {\dots};
					\node[above] (v4) at (4.5,0) {};
					\draw[fill] (4.5,0)  circle (.05);
					\node[above] (v5) at (6,0) {};
					\draw[fill] (6,0)  circle (.05);
					\draw[thick, blue, -latex] (0.15,0) -- (1.35,0);
					\draw[thick, bunired, -latex] (1.65,0) -- (2.5,0);
					\draw[thick, brown, -latex] (3.4,0) -- (4.35,0);
					\draw[thick, bunired, -latex] (4.65,0) -- (5.85,0);
					\node[above] (v0) at (0,-1) {};
					\draw[fill] (0,-1)  circle (.05);
					\node[above] (v1) at (1.5,-1) {};
					\draw[fill] (1.5,-1)  circle (.05);
					\node[] at (3,-1) {\dots};
					\node[above] (v4) at (4.5,-1) {};
					\draw[fill] (4.5,-1)  circle (.05);
					\node[above] (v5) at (6,-1) {};
					\draw[fill] (6,-1)  circle (.05);
					\draw[thick, bunired, -latex] (0.15,-1) -- (1.35,-1);
					\draw[thick, brown, -latex] (1.65,-1) -- (2.5,-1);
					\draw[thick, bunired, -latex] (3.4,-1) -- (4.35,-1);
					\draw[thick, blue, -latex] (4.65,-1) -- (5.85,-1);
					\draw[thick, blue, -latex] (0,-0.15) -- (0,-0.85);
					\draw[thick, bunired, -latex] (1.5,-0.15) -- (1.5,-0.85);
					\draw[thick, bunired, -latex] (4.5,-0.15) -- (4.5,-0.85);
					\draw[thick, blue, -latex] (6,-0.15) -- (6,-0.85);
			\end{tikzpicture}
     \caption{Decomposition in to dynamical modules of $\tI_n\times\tI_1$.}
     \label{fig:gridI_nxI1}
 \end{figure}
 The simplicial complex $X(\tI_n\times \tI_1)$ is then homotopy equivalent, in virtue of Theorem~\ref{thm:module}, to an iterated join:
 \[
 X(\tI_n\times \tI_1)\cong X(\tA_2 \sqcup \tA_3 \sqcup \dots \sqcup \tA_3 \sqcup \tA_2)\simeq X(\tA_2)^{*2}*X(\tA_3)^{*(n-1)} \ .
 \]
 As $X(\tA_2)\simeq X(\tA_3)$, and their geometric realisation is~the $0$-dimensional sphere, we get $X(\tI_n\times \tI_1)\simeq S^n$.
 
 Assume now both $n,m\geq 2$. Then, up to reversing all the orientations, we get the graph illustrated in Figure~\ref{fig:coherent grinds n x m}. In particular, in the decomposition in to dynamical modules, there is a module which is isomorphic to $\tA_4$; hence,  $X(\tI_n\times \tI_m)$ is homotopy equivalent to $X(\tA_4)\ast Y$, where $Y = X(C(\tA_4))$ -- see Proposition~\ref{prop:regions}. 
 As the multipath complex $X(\tA_4)$ is contractible, we get that also  $X(\tI_n\times \tI_m)$ is contractible, concluding the proof.
 \begin{figure}
     \centering
\begin{tikzpicture}[scale=1.0][baseline=(current bounding box.center)]
					\tikzstyle{point}=[circle,thick,draw=black,fill=black,inner sep=0pt,minimum width=2pt,minimum height=2pt]
					\tikzstyle{arc}=[shorten >= 8pt,shorten <= 8pt,->, thick]
					\node[above] (v0) at (0,1) {};
					\draw[fill] (0,1)  circle (.05);
					\node[above] (v1) at (1.5,1) {};
					\draw[fill] (1.5,1)  circle (.05);
					\node[above] (v2) at (3,1) {};
					\draw[fill] (3,1)  circle (.05);
					\node at (5.15,1) {$\dots$};
					\node at (3.75,0) {$\dots$};
					\node at (2.75,-1) {$\dots$};
					\node at (0.75,-2) {$\dots$};

					\node[above] (v2) at (4.5,1) {};
					\draw[fill] (4.5,1)  circle (.05);
					
					\node[above] (w0) at (0,-1) {};
					\draw[fill] (0,-1)  circle (.05);
					\node[above] (w2) at (1.5,-1) {};
					\draw[fill] (1.5,-1)  circle (.05);

					\node[above] (w0) at (0,-2) {};
					\draw[fill] (0,-2)  circle (.05);
					
					\draw[thick, blue, -latex] (0.15,1) -- (1.35,1);
					\draw[thick, bunired, latex-] (2.85,1) -- (1.65,1);
					\draw[thick, brown, -latex]  (3.15,1) -- (4.35,1) ;
					\node[above] (v0) at (0,0) {};
					\draw[fill] (0,0)  circle (.05);
					\node[above] (v1) at (1.5,0) {};
					\draw[fill] (1.5,0)  circle (.05);
					\node[above] (v2) at (3,0) {};
					\draw[fill] (3,0)  circle (.05);
					\draw[thick, bunired, -latex] (0.15,0) -- (1.35,0);
					\draw[thick, brown, latex-] (2.85,0) -- (1.65,0);
					\draw[thick, brown, -latex]  (0.15,-1) -- (1.35,-1) ;
					\draw[thick, blue, -latex] (0,0.85) -- (0,0.15);
					\draw[thick, bunired, -latex] (1.5,0.85) -- (1.5,0.15);
					\draw[thick, brown, -latex] (3,0.85) -- (3,0.15);
					\draw[thick, bunired, latex-] (0.0,-0.85) -- (0.0,-0.15);
					\draw[thick, brown, -latex] (0.0,-1.15) -- (0.0,-1.85);
					\draw[thick, brown, latex-] (1.5,-0.85) -- (1.5,-0.15);
			\end{tikzpicture}
     \caption{Part of the decomposition of $\tI_n\times \tI_m$ in to modules; in blue an $\tA_2$ component, in red an $\tA_4$ component.}
     \label{fig:coherent grinds n x m}
 \end{figure}
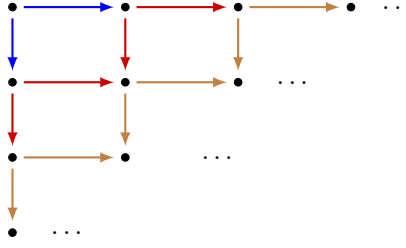
 \end{proof}

\begin{rem}\label{rem:noKunneth}
By Proposition~\ref{prop:coherent products}, although the homotopy type of $\tI_n$ is trivial, products of type $\tI_n\times \tI_1$ yield topological spheres. This implies that we cannot expect a K\"unneth-type formula for multipath cohomology.
\end{rem}
Now, we consider  another simple, yet interesting case: $\tA_n \times \tI_m$. First we recall that
a \emph{tree} is an undirected graph in which every two vertices are connected by exactly one path. 
A \emph{caterpillar} graph $\tG_n(m_1,\dots,m_n)$ is a tree consisting of a path on~$n$ vertices $v_1,\dots,v_n$, such that 
every vertex $v_i$ is connected to exactly $m_i$ other distinct vertices. 
An example of caterpillar graphs is given in Figure~\ref{fig:caterpillarT}.
Note that the homotopy-type of the matching complex of caterpillar graphs has been determined in \cite[Theorem~5.13]{caterpillar}.
\begin{figure}[H]
    \centering
    \begin{tikzpicture}
    [baseline=(current bounding box.center)]
				\tikzstyle{point}=[circle,thick,draw=black,fill=black,inner sep=0pt,minimum width=2pt,minimum height=2pt]
					\tikzstyle{arc}=[shorten >= 8pt,shorten <= 8pt,->, thick]

					\node[above] (w0) at (4.5,1) {$v_2$};
					\draw[fill] (4.5,1)  circle (.05);
					\node[above] (w1) at (6.0,1) {$v_3$};
					\draw[fill] (6.0,1)  circle (.05);
					\node[above] (w2) at (7.5,1) {$v_{n-2}$};
					\draw[fill] (7.5,1)  circle (.05);
					\node[above] () at (9.0,1) {$v_{n-1}$};
					\draw[fill] (9.0,1)  circle (.05);
					\node[above] () at (10.5,1) {$v_{n}$};
					\draw[fill] (10.5,1)  circle (.05);

					\node[above] (v2) at (3,1) {$v_1$};
					\draw[fill] (3,1)  circle (.05);

					\draw[thick, bunired]  (3.15,1) -- (4.35,1) ;
					\draw[thick, bunired] (5.85,1) -- (4.65,1);
					\draw[thick, bunired]  (8.75,1) -- (7.65,1) ;
					\draw[thick, bunired]  (9.15,1) -- (10.35,1) ;

					\node[above] (w0) at (4.5,0) {};
					\draw[fill] (4.5,0)  circle (.05);

					\node[above] (v1) at (1.5,0) {};

					\draw[fill] (9,0)  circle (.05);
					\node at (6.75,1) {$\dots$};

					\draw[thick, bunired] (9,0.85) -- (9,0.15);
					\draw[thick, bunired] (4.5,0.85) -- (4.5,0.15);
			\end{tikzpicture}
    \caption{A caterpillar graph $\tG_n(0,1,\dots,1,0) = \tG_{n-2}(2,\dots,2)$.}
    \label{fig:caterpillarT}
\end{figure}
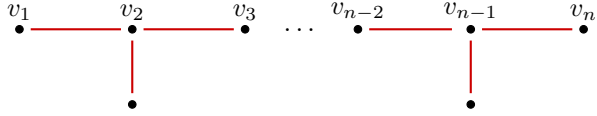

We need the homotopy type of matching complexes of some specific types of caterpillar graphs; namely, caterpillar graphs of type $\tG_{2n+1}(0,1,0\dots)$ with a single leg at each vertex in even position, and $\tG_n(1,0,1,\dots)$.
For $k\geq 1$, let $L_k(a_1,\dots,a_k)$ denote the sum
\[
L_k(a_1,\dots,a_k) = \sum_{i=1}^{k} a_i + \sum_{\substack{l=2,\dots, k, \\ 1\leq i_1<i_2<\dots<i_l\leq k}} (i_2-i_1)(i_3-i_2)\cdots(i_l-i_{l-1})a_{i_1}a_{i_2}\cdots a_{i_l} \ .
\]
The homotopy type of matching complexes of caterpillar graphs is then given as follows:

\begin{thm}[{\cite[Theorem~5.16]{caterpillar}}]\label{thm:caterpillarsimple}
Consider the caterpillar graph $\tG_{2k-1}(m_1,0,m_2,0,\dots,m_{k-1},0,m_{k})$ for $k\in\bN$, $m_i >0$. Then, the homotopy type of the associated matching complex is given by
\[
M(\tG_{2k-1}(m_1,0,m_2,0,\dots,m_{k-1},0,m_{k}))\simeq \bigvee_{L_k(a_1,\dots,a_k)}S^{k-1} \ ,
\]
where $a_i=m_i-1$ for $i=1,\dots,k$.
\end{thm}

A straightforward application of Theorem~\ref{thm:caterpillarsimple} is the following computation:

\begin{lem}\label{lem:catalt}
Consider the caterpillar graph $\tG_{s}(1,0,1,\dots)$ on $s\geq 2$ central vertices, endowed with the alternating orientation as illustrated in Figure~\ref{fig:grids n x m} (blue part). Then, the homotopy type of the multipath complex is given by
\[ X(\tG_{s}(1,0,1,\dots)) \simeq \begin{cases}
S^{\frac{s}{2}-1} & s \text{ even}, \\
\ast &\text{otherwise}
\end{cases}\]
and it is either contractible or a sphere.
\end{lem}

\begin{proof}
When $s$ is even, the caterpillar graph $\tG_s(1,0,1,\dots)$ can be seen as the caterpillar graph $\tG_{s-1}(1,0,1,\dots,1,0,2)$ on $s-1$ central vertices. The $m_1,\dots,m_k$ appearing in the statement of Theorem~\ref{thm:caterpillarsimple} are, in this case, all equal to $1$. When $s$ is odd, the sequence $(a_1,\dots,a_{s})$ is just the sequence $(0,\dots,0)$. When $s$ is even we have that $(a_1,\dots,a_{s-1})$ is the sequence $(0,\dots,0,1)$. Therefore, for $s$ odd, $L(0,\dots,0)=0$, whereas, for $s$ even, $L(0,\dots,0,1)=1$. The statement now follows from Theorem~\ref{thm:caterpillarsimple}.
\end{proof}

The computation of the homotopy type of matching complexes of caterpillar graphs is usually complicated; when the strings have a predictable pattern of zeros, computations might be carried on by looking at the $L_k$ polynomials. For example, we have the following computation, needed later, whose proof cannot be directly derived from Theorem~\ref{thm:caterpillarsimple};

\begin{lem}\label{lem:techlem2}
Assume $t_1 = 1 $, $t_{i} = 0$ for each $i$ such that $k >i > 1$, and $t_{k} \in \{ 0 ,1 \}$.
Then, \[ L_{k}(t_1, ... , t_{k}) = \begin{cases}
k + 1 & t_k = 1 \\ 1 & t_k = 0
\end{cases}\ ,\]
for all $k>3$.
\end{lem}
\begin{proof}
The statement follows, using the relation~\cite[Equation~(1)]{caterpillar}, by induction. 
\end{proof}

Set $\tS_1\coloneqq \tG_2(2,0)= \tG_{1}(3)$ and let $\tS_{n}\coloneqq \tG_{2n-1}(2,0,1,0,\dots,0,1,0,2)$ be the caterpillar graph with a single leg at each internal vertex in odd position, endowed with an alternating orientation (i.e.~all vertices are either sources or sinks).
Let $\tC_n$ be the caterpillar graph $\tG_{n+1}(2,0,1,0,1,...)$ where $0$ and $1$ alternate along the sequence, endowed with an alternating orientation; note that we have~$\tS_n = \tC_{2n-1}$.

\begin{lem}\label{lem:caterpillar20101}
 We have the following homotopy equivalence
\[ X(\tC_{n}) \simeq M(\tG_{n+1}(2,0,1,0,1,...)) \simeq \begin{cases}
S^{k-1} & n = 2k-2\\
\bigvee^{k+1} S^{k-1} & n = 2k-1
\end{cases}\]
where $M(\tG)$ denotes the matching complex.
In particular, $X(\tC_{n})$ is a wedge of spheres.
\end{lem}

\begin{proof}
The alternating orientation on $\tC_n$ implies that $\tC_n$ is a stable dynamical region; hence, by Lemma~\ref{lem:alternatingmultipath}, we have the homotopy equivalence $X(\tC_{n}) \simeq M(\tG_{n+1}(2,0,1,0,1,...))$ with the matching complex. 
Then, the statement follows directly from Theorem~\ref{thm:caterpillarsimple} and Lemma~\ref{lem:techlem2}.
\end{proof}

We can now compute the homotopy type of the multipath complex of grids $\tA_n \times \tI_m$.

\begin{prop}
Let $n, m$ be positive integers, then
\[ X(\tA_n \times \tI_m) \simeq M(\tG_{n+1}(1,...,1))^{\ast(m-1)}\ast X(\tA_n \times \tI_1) \ .\]
In particular, $X(\tA_n \times \tI_m)$ is contractible if $n$ is even, and a sphere of dimension  $(m-1)\frac{n+1}{2} + n$ when $n$ is odd. 
\end{prop}

\begin{proof}
The product $ \tA_n \times \tI_m $ has a decomposition in to dynamical modules featuring $m-1$ copies of caterpillar graphs of type $\tG_{n+1}(1,\cdots,1)$, and two copies of caterpillar graphs of type $\tG_{n+1}(1,0,1,\dots)$, all with alternating orientations -- see also Figure~\ref{fig:grids n x m}.  By Lemma~\ref{lem:alternatingmultipath} and Theorem~\ref{thm:module}, $X(\tA_n \times \tI_m)$ decomposes as $M(\tG_{n+1}(1,...,1))^{m-1}\ast X(\tA_n \times \tI_1)$. By \cite[Corollary~5.12]{caterpillar}, $M(\tG_{n+1}(1,...,1))$ is contractible when $n$ is even, and a sphere otherwise, hence $M(\tG_{n+1}(1,...,1))^{\ast (m-1)}$ is contractible when $n$ is even, and a sphere otherwise. 

Observe that $X(\tA_n\times \tI_1)$ is homotopic to $M(\tG_{n}(1,0,1,\dots,2))\ast M(\tG_{n}(1,0,1,\dots,2))$  when $n$ is odd, and homotopic to $M(\tG_{n+1}(1,0,1,\dots,1))\ast M(\tS_{\frac{n}{2}})$ when $n$ is even. By Lemma~\ref{lem:catalt}, $M(\tG_{n+1}(1,0,1,\dots,1))$ is contractible, and  $M(\tG_{n}(1,0,1,\dots,2))\ast M(\tG_{n}(1,0,1,\dots,2))$ is a sphere of dimension $2\tfrac{n-1}{2} + 1 = n$, hence $X(\tA_n\times \tI_1)$ is contractible when $n$ is even, and a sphere when $n$ is odd. 
\end{proof}
 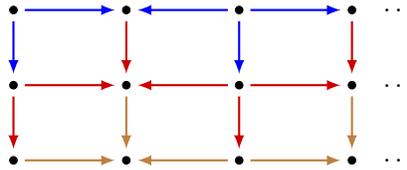
\begin{figure}[h]
     \centering
\begin{tikzpicture}[scale=1.0][baseline=(current bounding box.center)]
					\tikzstyle{point}=[circle,thick,draw=black,fill=black,inner sep=0pt,minimum width=2pt,minimum height=2pt]
					\tikzstyle{arc}=[shorten >= 8pt,shorten <= 8pt,->, thick]
					\node[above] (v0) at (0,1) {};
					\draw[fill] (0,1)  circle (.05);
					\node[above] (v1) at (1.5,1) {};
					\draw[fill] (1.5,1)  circle (.05);
					\node[above] (v2) at (3,1) {};
					\draw[fill] (3,1)  circle (.05);
					\node at (5.15,1) {$\dots$};
					\node at (5.15,0) {$\dots$};
					\node at (5.15,-1) {$\dots$};

					\node[above] (v2) at (4.5,1) {};
					\draw[fill] (4.5,1)  circle (.05);
					
					\node[above] (w0) at (0,-1) {};
					\draw[fill] (0,-1)  circle (.05);
					\node[above] (w1) at (3.0,-1) {};
					\draw[fill] (3.0,-1)  circle (.05);
					\node[above] (w2) at (1.5,-1) {};
					\draw[fill] (1.5,-1)  circle (.05);
					\node[above] (w5) at (4.5,0) {};
					\draw[fill] (4.5,0)  circle (.05);
					\node[above] (w0) at (4.5,-1) {};
					\draw[fill] (4.5,-1)  circle (.05);
					
					\draw[thick, blue, -latex] (0.15,1) -- (1.35,1);
					\draw[thick, blue, -latex] (2.85,1) -- (1.65,1);
					\draw[thick, blue, -latex]  (3.15,1) -- (4.35,1) ;
					\node[above] (v0) at (0,0) {};
					\draw[fill] (0,0)  circle (.05);
					\node[above] (v1) at (1.5,0) {};
					\draw[fill] (1.5,0)  circle (.05);
					\node[above] (v2) at (3,0) {};
					\draw[fill] (3,0)  circle (.05);
					\draw[thick, bunired, -latex] (0.15,0) -- (1.35,0);
					\draw[thick, bunired, -latex] (3.15,0) -- (4.35,0);
					\draw[thick, brown, -latex] (3.15,-1) -- (4.35,-1);
					\draw[thick, brown, -latex] (2.85,-1) -- (1.65,-1);
					\draw[thick, bunired, -latex] (2.85,0) -- (1.65,0);
					\draw[thick, brown, -latex]  (0.15,-1) -- (1.35,-1) ;
					\draw[thick, blue, -latex] (0,0.85) -- (0,0.15);
					\draw[thick, bunired, -latex] (1.5,0.85) -- (1.5,0.15);
					\draw[thick, blue, -latex] (3,0.85) -- (3,0.15);
					\draw[thick, bunired, -latex] (4.5,0.85) -- (4.5,0.15);
					\draw[thick, bunired, latex-] (0.0,-0.85) -- (0.0,-0.15);
					\draw[thick, bunired, latex-] (3.0,-0.85) -- (3.0,-0.15);
					\draw[thick, brown, latex-] (4.5,-0.85) -- (4.5,-0.15);
					\draw[thick, brown, latex-] (1.5,-0.85) -- (1.5,-0.15);
			\end{tikzpicture}
     \caption{Part of the decomposition of $\tA_n\times \tI_2$ in to dynamical modules.}
     \label{fig:grids n x m}
 \end{figure}

We proceed with the computation of the (homotopy type of the) multipath complex associated to general small grids of type~$\tL\times \tI_1$, for a linear digraph $\tL$. We may assume  $\tL\neq \tI_n,\tA_n$, since we already analysed these cases.
Assume first that $\tL$ decomposes in to an unstable dynamical region of positive size, followed by another linear graph $\tL'$. In other words, we have a coherent linear graph~$\tI_n$ ($n-1$ being the size of the unstable dynamical region) followed by an alternating linear graph $\tA_m$, and so on -- see also Figure~\ref{fig:linear even alternating}.

\begin{prop}
Consider the graph $\tL$ on $n+m-1$ vertices given by a coherent linear graph $\tI_n$ followed by an alternating graph $\tA_m$. Then, the homotopy type of $X(\tL \times \tI_1)$ depends on the parity of $m$ as follows:
\[ X(\tL \times \tI_1) \simeq \begin{cases}
\bigvee^{q(m)}S^{n + m + 3} & m \text{ even},\\
\bigvee^{\frac{m +3}{2}q(m+1)} S^{n + m + 3} & m \text{ odd},
\end{cases}\]
where $q(m) = 2^{\frac{m+2}{2}}$.
\end{prop}

\begin{proof}
By Theorem~\ref{thm:module}, we can decompose $\tL\times \tI_1$ into modules:
one copy of $\tA_2$, $(n-2)$ copies of $\tA_3$, and two caterpillar graphs $\tC_1$ and $\tC_2$, oriented as illustrated in Figure~\ref{fig:linear even alternating}. Hence, the homotopy type of $X(\tL\times \tI_1)$ is given by:
\[
X(\tL\times \tI_1)\simeq X(\tA_2)*X(\tA_3)^{\ast (n-2)} * X(\tC_1)* X(\tC_2) \ ,
\]
where $\tC_{1} = \tG_{m+3}(1,0,0,1,0,...)$, while $\tC_2 = \tG_{m+1}(2,0,1,0,1,...)$. (Note that for $m=0$, $X(\tC_1) = X(\tA_3)$ and $X(\tC_2) = X(\tA_2)$, which is coherent with our computations for $\tI_n\times \tI_1$.)
While the precise homotopy types of the matching complex of the caterpillar graphs $ \tG_{m+3}(1,0,0,1,0,...)$ and $\tG_{m+1}(2,0,1,0,1,...)$ depend on the parity of~$m$, in any case they are wedges of spheres. 
By \cite[Theorem~5.13]{caterpillar}, 
we have
\[X(\tC_1) \simeq M(\tG_{m+3}(1,0,0,1,0,...))\simeq  \bigvee^{s(m)}S^{\left\lceil\tfrac{m+3}{2}\right\rceil} \ , \]
where $s(m)=2^{\left\lfloor\tfrac{m+3}{2}\right\rfloor}$.
Directly from Lemma~\ref{lem:caterpillar20101} we have
\[ X(\tC_{2}) \simeq M(\tG_{m+1}(2,0,1,0,1,...))\simeq \begin{cases}
S^{k-1} & m = 2k-2,\\
\bigvee^{k+1} S^{k-1} & m = 2k - 1.
\end{cases}\]
The statement now follows from the properties of joins and  wedges of spheres.
\end{proof}

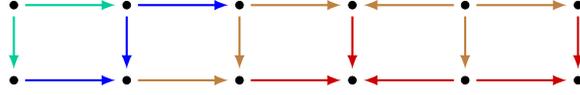
\begin{figure}
    \centering
\begin{tikzpicture}[scale=1.0][baseline=(current bounding box.center)]
					\tikzstyle{point}=[circle,thick,draw=black,fill=black,inner sep=0pt,minimum width=2pt,minimum height=2pt]
					\tikzstyle{arc}=[shorten >= 8pt,shorten <= 8pt,->, thick]
					\node[above] (v0) at (0,1) {};
					\draw[fill] (0,1)  circle (.05);
					\node[above] (w0) at (4.5,1) {};
					\draw[fill] (4.5,1)  circle (.05);
					\node[above] (w1) at (6.0,1) {};
					\draw[fill] (6.0,1)  circle (.05);
					\node[above] (w2) at (7.5,1) {};
					\draw[fill] (7.5,1)  circle (.05);
					\node[above] (v1) at (1.5,1) {};
					\draw[fill] (1.5,1)  circle (.05);
					\node[above] (v2) at (3,1) {};
					\draw[fill] (3,1)  circle (.05);
					\draw[thick, caribbeangreen, -latex] (0.15,1) -- (1.35,1);
					\draw[thick, blue, latex-] (2.85,1) -- (1.65,1);
					\draw[thick, brown, -latex]  (3.15,1) -- (4.35,1) ;
					\draw[thick, brown, -latex] (5.85,1) -- (4.65,1);
					\draw[thick, brown, -latex]  (6.15,1) -- (7.35,1) ;
					\node[above] (v0) at (0,0) {};
					\draw[fill] (0,0)  circle (.05);
					\node[above] (w0) at (4.5,0) {};
					\draw[fill] (4.5,0)  circle (.05);
					\node[above] (w1) at (6.0,0) {};
					\draw[fill] (6.0,0)  circle (.05);
					\node[above] (w2) at (7.5,0) {};
					\draw[fill] (7.5,0)  circle (.05);
					\node[above] (v1) at (1.5,0) {};
					\draw[fill] (1.5,0)  circle (.05);
					\node[above] (v2) at (3,0) {};
					\draw[fill] (3,0)  circle (.05);
					\draw[thick, blue, -latex] (0.15,0) -- (1.35,0);
					\draw[thick, brown, latex-] (2.85,0) -- (1.65,0);
					\draw[thick, bunired, -latex]  (3.15,0) -- (4.35,0) ;
					\draw[thick, bunired, -latex] (5.85,0) -- (4.65,0);
					\draw[thick, bunired, -latex]  (6.15,0) -- (7.35,0) ;
					\draw[thick, caribbeangreen, -latex] (0,0.85) -- (0,0.15);
					\draw[thick, blue, -latex] (1.5,0.85) -- (1.5,0.15);
					\draw[thick, brown, -latex] (3,0.85) -- (3,0.15);
					\draw[thick, bunired, -latex] (4.5,0.85) -- (4.5,0.15);
					\draw[thick, brown, -latex] (6,0.85) -- (6,0.15);
					\draw[thick, bunired, -latex] (7.5,0.85) -- (7.5,0.15);
			\end{tikzpicture}    
			\caption{Linear graph consisting of a graph $\tI_3$ followed by an $\tA_2$.}
    \label{fig:linear even alternating}
\end{figure}

More generally, given any oriented linear graph $\tL$, one can decompose it in to joins of multipath complexes associated to caterpillar graphs endowed with alternating orientations. The next proposition follows;
\begin{prop} \label{prop:weak_dec_LxI}
If $\tL$ is a linear graph, then $\tL\times \tI_1$  decomposes into dynamical modules that are  caterpillar graphs (with  alternating orientations).
\end{prop}
\proof 
We proceed by induction on the number of edges $n$. 
If $\tL_{n}$ is a linear graph on $n$ edges, the statement holds true for $\tL_0$, and it is easy to prove for $\tL_1=\tI_1$. 
We now analyse what happens to the grid $\tL_n\times \tI_1$ when adding an (oriented) edge, obtaining $\tL_{n+1}\times \tI_1$. Up to reversing the orientation of all edges in our grid, we can restrict to two different cases, as illustrated in Figures~\ref{fig:Case1} and~\ref{fig:Case2}. 
\begin{figure}[H]
    \centering
\begin{tikzpicture}[scale=1.0][baseline=(current bounding box.center)]
					\tikzstyle{point}=[circle,thick,draw=black,fill=black,inner sep=0pt,minimum width=2pt,minimum height=2pt]
					\tikzstyle{arc}=[shorten >= 8pt,shorten <= 8pt,->, thick]
                    \node[] (w0) at (-0.5,1) {\dots};
					\node[above] (v0) at (0,1) {};
					\draw[fill] (0,1)  circle (.05);
					\node[above] (v1) at (1.5,1) {};
					\draw[fill] (1.5,1)  circle (.05);
					\node[above] (v2) at (3,1) {};
					\draw[fill] (3,1)  circle (.05);
					\draw[thick, green, -latex] (0.15,1) -- (1.35,1);
					\draw[thick, bunired, latex-] (2.85,1) -- (1.65,1);
                    \node[] (w0) at (-0.5,0) {\dots};
					\node[above] (v0) at (0,0) {};
					\draw[fill] (0,0)  circle (.05);
					\node[above] (v1) at (1.5,0) {};
					\draw[fill] (1.5,0)  circle (.05);
					\node[above] (v2) at (3,0) {};
					\draw[fill] (3,0)  circle (.05);
					\draw[thick, blue, -latex] (0.15,0) -- (1.35,0);
					\draw[thick, brown, latex-] (2.85,0) -- (1.65,0);
					\draw[thick, green, -latex] (0,0.85) -- (0,0.15);
					\draw[thick, blue, -latex] (1.5,0.85) -- (1.5,0.15);
					\draw[thick, brown, -latex] (3,0.85) -- (3,0.15);
			\end{tikzpicture}    
\caption{First Case: An edge is glued to $\tL_n$ in a coherent way. }
    \label{fig:Case1}
\end{figure}
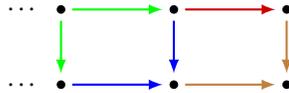
The blue edges and the green edges in both figures belong to different dynamical modules of $\tL_n\times \tI_1$; these are both, by the inductive hypothesis, caterpillar graphs with an alternating orientation. In the case illustrated in Figure~\ref{fig:Case1}, the module decomposition of $\tL_{n+1}\times \tI_1$ is obtained as follows;  one module is obtained by  adding the red edge to the module of $\tL_n\times \tI_1$ featuring the blue edges (yielding a caterpillar graph with an alternating orientation), all the other modules of $\tL_n\times \tI_1$ remain unaffected, and, in addition to those, there is a further caterpillar graph of type $\tA_{2}$ (in brown) appearing in the decomposition.  \begin{figure}[H]
    \centering
\begin{tikzpicture}[scale=1.0][baseline=(current bounding box.center)]
					\tikzstyle{point}=[circle,thick,draw=black,fill=black,inner sep=0pt,minimum width=2pt,minimum height=2pt]
					\tikzstyle{arc}=[shorten >= 8pt,shorten <= 8pt,->, thick]
                    \node[] (w0) at (-0.5,1) {\dots};
					\node[above] (v0) at (0,1) {};
					\draw[fill] (0,1)  circle (.05);
					\node[above] (v1) at (1.5,1) {};
					\draw[fill] (1.5,1)  circle (.05);
					\node[above] (v2) at (3,1) {};
					\draw[fill] (3,1)  circle (.05);
					\draw[thick, green, -latex] (0.15,1) -- (1.35,1);
					\draw[thick, darkgreen, latex-] (1.65,1) -- (2.85,1);
                    \node[] (w0) at (-0.5,0) {\dots};
					\node[above] (v0) at (0,0) {};
					\draw[fill] (0,0)  circle (.05);
					\node[above] (v1) at (1.5,0) {};
					\draw[fill] (1.5,0)  circle (.05);
					\node[above] (v2) at (3,0) {};
					\draw[fill] (3,0)  circle (.05);
					\draw[thick, blue, -latex] (0.15,0) -- (1.35,0);
					\draw[thick, bunired, latex-] (1.65,0) -- (2.85,0);
					\draw[thick, green, -latex] (0,0.85) -- (0,0.15);
					\draw[thick, blue, -latex] (1.5,0.85) -- (1.5,0.15);
					\draw[thick, darkgreen, -latex] (3,0.85) -- (3,0.15);
			\end{tikzpicture}   
			\caption{Second Case: An edge is glued to $\tL_n$ in a non-coherent way. }
    \label{fig:Case2}
\end{figure}
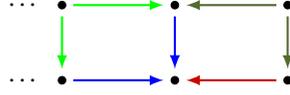
Similarly, in the second case (see Figure~\ref{fig:Case2}), the dark green edges are added to the module of $\tL_n\times \tI_1$ in light green, and the isolated red edge is added to the blue module of $\tL_n\times \tI_1$; the other modules  of $\tL_n\times \tI_1$ remain unaffected,  concluding the proof. 
 \endproof
\begin{cor}
If $\tL$ is a linear graph, then $X(\tL\times \tI_1)$ is either contractible or a wedge of spheres.
\end{cor}

\begin{proof}
Since the homotopy type of the multipath complex of a caterpillar graph with an alternating orientation is a wedge of spheres, the result follows from Proposition~\ref{prop:weak_dec_LxI}. 
\end{proof}

We remark that, reasoning as in the proof of Proposition \ref{prop:weak_dec_LxI}, it is possible to compute iteratively the number and dimension of spheres appearing in $X(\tL\times \tI_1)$.

\section{Multipath complexes of transitive tournaments}

The techniques developed in the previous section are ineffective in the case of alternating digraphs or transitive tournaments. Transitive tournaments, in fact, are dynamical modules themselves, and do not admit a smaller decomposition. Nonetheless, using techniques borrowed from combinatorial topology, we can yet compute their homotopy types. The aim of this section is to show that the homotopy type of the multipath complexes associated to transitive tournaments is also either contractible or a wedge of spheres.

Recall that  $\tT_n$ denotes the transitive tournament on  $n+1$ vertices, i.e.~the directed graph on vertices $0,\dots,n$  with directed  edges $(i,j)$ for all $i\leq j$; denote by $X(n)$ its associated multipath complex. The main result of the section is the following:

\begin{thm}\label{thm:cliqueshomotopy}
The multipath complex $X(n)$ of the transitive tournament $\tT_n$ is either contractible, or homotopy equivalent to a wedge of spheres.
\end{thm}

\begin{rem}\label{rem:torsion}
The matching complex of the complete graph on $7$ vertices has $3$-torsion \cite{Bouc} (compare with~\cite[Theorem~1.3 and Remark~1.4]{torsion}). By Theorem~\ref{thm:cliqueshomotopy}, the multipath complex of a transitive tournament is contractible or a wedge of spheres. On the other hand, the matching complex can be seen as a subcomplex of the multipath complex -- see also~\cite[Section~4]{SpriSecondo}. This means that, in the case of transitive tournaments, the cells added to the matching complex to obtain the multipath complex kill the torsion.
\end{rem}

The proof of  Theorem~\ref{thm:cliqueshomotopy} will heavily rely on the following lemma:

\begin{lem}[{\cite[Lemma~10.4(ii)]{BjorTopMeth}}]\label{lem:BjornerLemma}
Suppose that $X$ is a simplicial complex which can be written as the union of subcomplexes $X_{0},\dots, X_{n}$ such that:
\begin{enumerate}[label = (\alph*)]
\item $X_i$ is contractible for each $i= 0,\dots, n$, and
\item $X_i\cap X_j \subseteq X_0$ for all $i,j\in \{ 1,..,n\}$.
\end{enumerate}
Then, we have a homotopy equivalence
\[ X \simeq \bigvee_{i=1}^n \Sigma(X_{0} \cap X_{i}) \ ,\]
where $\Sigma(X_{0} \cap X_{i})$ denotes the topological suspension of $(X_{0} \cap X_{i})$.
\end{lem}

We remark that, by convention, $\Sigma \emptyset = S^{0} $, hence the suspension on the empty set is the $0$-dimensional sphere.

For a digraph~$\tG$, the \emph{digraph suspension} $\Sigma(\tG)$ is defined as the digraph with vertices $V(\tG)\cup \{p,q\}$, with $p,q\notin V(\tG)$, and edge set the edges of $\tG$ along with edges $(v,p)$ and $(v,q)$, for all $v$ in $V(\tG)$.
A straightforward application of Lemma~\ref{lem:BjornerLemma} allows us to compute the homotopy type of the digraph suspension in some cases. 

\begin{prop}\label{prop:suspension}
Let $\tG$ be a connected digraph with at least a vertex $v$ of outdegree $0$, and non-zero indegree. Then, there is a homotopy equivalence
\[
X(\Sigma\tG)\simeq \Sigma X(\tG)
\]
between the multipath complex of the digraph suspension and the topological suspension of the multipath complex of~$\tG$.
\end{prop}

\begin{proof}
Let $p, q$ be the added vertices of $V(\Sigma\tG) \setminus V(\tG)$.
Consider the decomposition of the simplicial complex $X(\Sigma\tG)$ given as follows; $X_0$ is the subcomplex of $X(\tG)$ spanned by all multipaths containing the edge~$(v,p)$, and $X_1$  the subcomplex of $X(\tG)$ spanned by all multipaths containing the edge~$(v,q)$. Since the outdegree of $v$ in $\tG$ is zero, it is clear that $X_0\cup X_1=X(\Sigma\tG)$. Moreover, both $X_0$ and $X_1$ are contractible. The intersection $X_0\cap X_1$ is the multipath complex of~$\tG$, hence $X(\Sigma\tG)\simeq \Sigma X(\tG)$.
\end{proof}

Before proceeding with the proof of Theorem~\ref{thm:cliqueshomotopy}, we need to introduce some more notation.

\begin{defn}
Consider the transitive tournament $\tT_n$  on vertices $0,\dots,n$.  For indices $0\leq i_1 < \dots < i_k \leq n$, denote by $\tT_n^{(i_1, \dots, i_k)}$ the subgraph of $\tT_n$ obtained by removing all edges of type $(i_j, h)$ for $j=1, \dots, k$ and $h\geq i_j$.  We call such subgraphs \emph{incomplete tournaments}.
\end{defn}

 Note that $\tT_n^{(n)} = \tT_n$. Further examples of incomplete tournaments can be found in Figures~\ref{fig:[5]^(3)}, \ref{fig:decomposition5-3} and~\ref{fig:small_incoplete}. Figure~\ref{fig:decomposition5-3} illustrates a decomposition of $X(\tT_5^{(3)})$ in to subcomplexes. 
 \begin{figure}[h]
     \centering
     \begin{tikzpicture}[scale=1.75, thick]
    \node (a) at (0,0) {};
    \node (b) at (.5,.866) {};
    \node (c) at (1,0) {};
    \node (d) at (.5,0.288) {};
    \node (e) at (0.5,-0.288) {};
    \node (f) at (0.5,-0.866) {};
    
    \node at (0,0) [below left] {$0$};
    \node at (.5,.866) [above] {$3$};
    \node at (1,0) [below right] {$2$};
    \node at (.5,0.288) [above left] {$1$};
    \node at (0.5,-0.288) [below left] {$4$};
    \node at (0.5,-0.866)[below] {$5$};

    \draw[black, fill] (a) circle (.025);
    \draw[black, fill] (b) circle (.025);
    \draw[black, fill] (c) circle (.025);
    \draw[black, fill] (d) circle (.025);
    \draw[black, fill] (e) circle (.025);
    \draw[black, fill] (f) circle (.025);

    \draw[-latex, bunired] (a) -- (b);
    \draw[-latex, bunired] (a) -- (c);
    \draw[-latex, bunired] (c) -- (b);
    \draw[-latex, bunired] (a) -- (d);
    \draw[latex-, bunired] (c) -- (d);
    \draw[-latex, bunired] (d) -- (b);

    \draw[-latex, bunired] (c) -- (f);
    \draw[-latex, bunired] (c) -- (e);
    \draw[-latex, bunired] (a) -- (f);
    \draw[-latex, bunired] (a) -- (e);
    \draw[-latex, bunired] (e) -- (f);
    \draw[-latex,  line width = 3, white] (d) -- (e);
    \draw[-latex, bunired] (d) -- (e);
    \draw[-latex, line width = 3, white] (d) .. controls +(.35,-.35) and +(.125,.25) .. (f);
    \draw[-latex, bunired] (d) .. controls +(.35,-.35) and +(.125,.25) .. (f);
     \end{tikzpicture}
     \caption{The incomplete tournament $\tT_5^{(3)}$.}
     \label{fig:[5]^(3)}
 \end{figure}
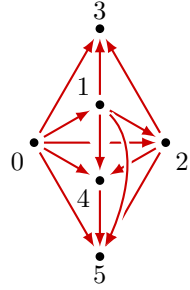
 
 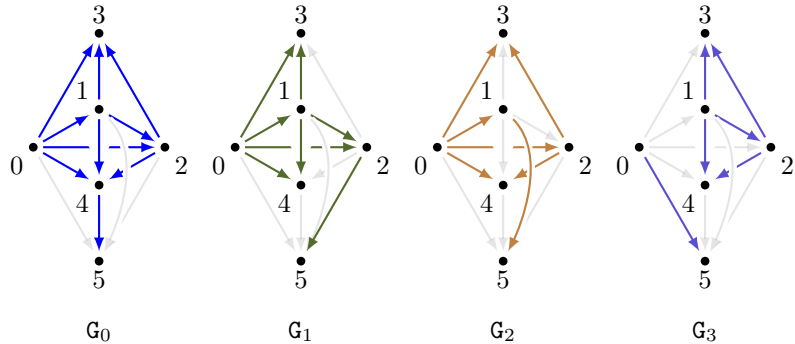
\begin{figure}
     \centering
     \begin{tikzpicture}[scale=1.75, thick]
    \node (a) at (0,0) {};
    \node (b) at (.5,.866) {};
    \node (c) at (1,0) {};
    \node (d) at (.5,0.288) {};
    \node (e) at (0.5,-0.288) {};
    \node (f) at (0.5,-0.866) {};
    
    \node at (0,0) [below left] {$0$};
    \node at (.5,.866) [above] {$3$};
    \node at (1,0) [below right] {$ 2$};
    \node at (.5,0.288) [above left] {$ 1$};
    \node at (0.5,-0.288) [below left] {$ 4$};
    \node at (0.5,-0.866)[below] {$ 5$};
    
    \node at (0.5,-1.25) [below] {$\tG_0$};
    
    \draw[black, fill] (a) circle (.025);
    \draw[black, fill] (b) circle (.025);
    \draw[black, fill] (c) circle (.025);
    \draw[black, fill] (d) circle (.025);
    \draw[black, fill] (e) circle (.025);
    \draw[black, fill] (f) circle (.025);

    \draw[-latex, blue] (a) -- (b);
    \draw[-latex, blue] (a) -- (c);
    \draw[-latex, blue] (c) -- (b);
    \draw[-latex, blue] (a) -- (d);
    \draw[-latex, blue] (d) -- (c);
    \draw[-latex, blue] (d) -- (b);

    \draw[-latex, opacity = .2, gray] (c) -- (f);
    \draw[-latex, blue] (c) -- (e);
    \draw[-latex, opacity = .2, gray] (a) -- (f);
    \draw[-latex, blue] (a) -- (e);
    \draw[-latex, blue] (e) -- (f);
    \draw[-latex,  line width = 3, white] (d) -- (e);
    \draw[-latex, blue] (d) -- (e);
    \draw[-latex, line width = 3, white] (d) .. controls +(.35,-.35) and +(.125,.25) .. (f);
    \draw[-latex, opacity = .2, gray] (d) .. controls +(.35,-.35) and +(.125,.25) .. (f);
     \end{tikzpicture}
     \begin{tikzpicture}[scale=1.75, thick]
    \node (a) at (0,0) {};
    \node (b) at (.5,.866) {};
    \node (c) at (1,0) {};
    \node (d) at (.5,0.288) {};
    \node (e) at (0.5,-0.288) {};
    \node (f) at (0.5,-0.866) {};
    
    \node at (0,0) [below left] {$ 0$};
    \node at (.5,.866) [above] {$ 3$};
    \node at (1,0) [below right] {$ 2$};
    \node at (.5,0.288) [above left] {$ 1$};
    \node at (0.5,-0.288) [below left] {$ 4$};
    \node at (0.5,-0.866)[below] {$ 5$};
    \node at (0.5,-1.25) [below] {$\tG_1$};
    
    \draw[black, fill] (a) circle (.025);
    \draw[black, fill] (b) circle (.025);
    \draw[black, fill] (c) circle (.025);
    \draw[black, fill] (d) circle (.025);
    \draw[black, fill] (e) circle (.025);
    \draw[black, fill] (f) circle (.025);

    \draw[-latex, darkgreen] (a) -- (b);
    \draw[-latex, darkgreen] (a) -- (c);
    \draw[-latex, opacity = .2, gray] (c) -- (b);
    \draw[-latex, darkgreen] (a) -- (d);
    \draw[latex-, darkgreen] (c) -- (d);
    \draw[-latex, darkgreen] (d) -- (b);

    \draw[-latex, opacity = .2, gray] (c) -- (e);
    \draw[-latex, opacity = .2, gray] (a) -- (f);
    \draw[-latex, darkgreen] (a) -- (e);
    \draw[-latex, opacity = .2, gray] (e) -- (f);
    \draw[-latex,  line width = 3, white] (d) -- (e);
    \draw[-latex, darkgreen] (d) -- (e);
    \draw[-latex, line width = 3, white] (d) .. controls +(.35,-.35) and +(.125,.25) .. (f);
    \draw[-latex, opacity = .2, gray] (d) .. controls +(.35,-.35) and +(.125,.25) .. (f);
    \draw[-latex, darkgreen] (c) -- (f);
     \end{tikzpicture}
       \begin{tikzpicture}[scale=1.75, thick]
    \node (a) at (0,0) {};
    \node (b) at (.5,.866) {};
    \node (c) at (1,0) {};
    \node (d) at (.5,0.288) {};
    \node (e) at (0.5,-0.288) {};
    \node (f) at (0.5,-0.866) {};
    
    \node at (0,0) [below left] {$ 0$};
    \node at (.5,.866) [above] {$ 3$};
    \node at (1,0) [below right] {$ 2$};
    \node at (.5,0.288) [above left] {$ 1$};
    \node at (0.5,-0.288) [below left] {$ 4$};
    \node at (0.5,-0.866)[below] {$ 5$};
    
    \node at (0.5,-1.25) [below] {$\tG_2$};
    
    \draw[black, fill] (a) circle (.025);
    \draw[black, fill] (b) circle (.025);
    \draw[black, fill] (c) circle (.025);
    \draw[black, fill] (d) circle (.025);
    \draw[black, fill] (e) circle (.025);
    \draw[black, fill] (f) circle (.025);

    \draw[-latex, brown] (a) -- (b);
    \draw[-latex, brown] (a) -- (c);
    \draw[-latex, brown] (c) -- (b);
    \draw[-latex, brown] (a) -- (d);
    \draw[latex-, opacity = .2, gray] (c) -- (d);
    \draw[-latex, opacity = .2, gray] (d) -- (b);

    \draw[-latex,  opacity = .2, gray] (c) -- (f);
    \draw[-latex, brown] (c) -- (e);
    \draw[-latex,  opacity = .2, gray] (a) -- (f);
    \draw[-latex, brown] (a) -- (e);
    \draw[-latex,  opacity = .2, gray] (e) -- (f);
    \draw[-latex,  line width = 3, white] (d) -- (e);
    \draw[-latex,  opacity = .2, gray] (d) -- (e);
    \draw[-latex, line width = 3, white] (d) .. controls +(.35,-.35) and +(.125,.25) .. (f);
    \draw[-latex, brown] (d) .. controls +(.35,-.35) and +(.125,.25) .. (f);
     \end{tikzpicture}
      \begin{tikzpicture}[scale=1.75, thick]
    \node (a) at (0,0) {};
    \node (b) at (.5,.866) {};
    \node (c) at (1,0) {};
    \node (d) at (.5,0.288) {};
    \node (e) at (0.5,-0.288) {};
    \node (f) at (0.5,-0.866) {};
    
    \node at (0,0) [below left] {$ 0$};
    \node at (.5,.866) [above] {$ 3$};
    \node at (1,0) [below right] {$ 2$};
    \node at (.5,0.288) [above left] {$ 1$};
    \node at (0.5,-0.288) [below left] {$ 4$};
    \node at (0.5,-0.866)[below] {$ 5$};

    \node at (0.5,-1.25) [below] {$\tG_3$};
    
    \draw[black, fill] (a) circle (.025);
    \draw[black, fill] (b) circle (.025);
    \draw[black, fill] (c) circle (.025);
    \draw[black, fill] (d) circle (.025);
    \draw[black, fill] (e) circle (.025);
    \draw[black, fill] (f) circle (.025);

    \draw[-latex, opacity = .2, gray] (a) -- (b);
    \draw[-latex, opacity = .2, gray] (a) -- (c);
    \draw[-latex, iris] (c) -- (b);
    \draw[-latex,  opacity = .2, gray] (a) -- (d);
    \draw[latex-, iris] (c) -- (d);
    \draw[-latex, iris] (d) -- (b);

    \draw[-latex,  opacity = .2, gray] (c) -- (f);
    \draw[-latex, iris] (c) -- (e);
    \draw[-latex, iris] (a) -- (f);
    \draw[-latex,  opacity = .2, gray] (a) -- (e);
    \draw[-latex,  opacity = .2, gray] (e) -- (f);
    \draw[-latex,  line width = 3, white] (d) -- (e);
    \draw[-latex, iris] (d) -- (e);
    \draw[-latex, line width = 3, white] (d) .. controls +(.35,-.35) and +(.125,.25) .. (f);
    \draw[-latex,  opacity = .2, gray] (d) .. controls +(.35,-.35) and +(.125,.25) .. (f);
     \end{tikzpicture}
     \caption{Decomposition of $\tG = \tT_5^{(3)}$.}
     \label{fig:decomposition5-3}
 \end{figure}

\begin{lem}\label{lem:SmallTournaments}
The multipath complex of each incomplete tournament of a transitive tournament on $2$, $3$, or $4$ vertices is empty, contractible or a wedge of spheres.
\end{lem}
\begin{proof}
The assertion follows by direct computation; see Figure~\ref{fig:small_incoplete}. The only nontrivial case is $\tT_3^{(2)}$, which is the digraph $\Sigma \tT_1$. Now, by Proposition~\ref{prop:suspension}, it follows that $X(\tT_3^{(2)})$ is contractible, concluding the computation.
\end{proof}

\begin{figure}[h]
    \centering
    \begin{tikzpicture}[scale = 1.5, thick]
    \node (a) at (0,0) {};
    \node (b) at (.5,.866) {};
    \node (c) at (1,0) {};
    
    \node at (0,0) [below left] {$ 0$};
    \node at (.5,.866) [above] {$ 1$};
    \node at (1,0) [below right] {$ 2$};

    \draw[black, fill] (a) circle (.035);
    \draw[black, fill] (b) circle (.035);
    \draw[black, fill] (c) circle (.035);
    
    \draw[-latex, bunired] (a) -- (b);
    \draw[-latex, bunired] (a) -- (c);
    \draw[-latex, bunired] (b) -- (c);
    
    \begin{scope}[shift = {+(3,0)}]
     \node (a) at (0,0) {};
    \node (b) at (.5,.866) {};
    \node (c) at (1,0) {};
    
    \node at (0,0) [below left] {$ 0$};
    \node at (.5,.866) [above] {$ 1$};
    \node at (1,0) [below right] {$ 2$};

    \draw[black, fill] (a) circle (.035);
    \draw[black, fill] (b) circle (.035);
    \draw[black, fill] (c) circle (.035);
    
    \draw[-latex, bunired] (a) -- (b);
    \draw[-latex, bunired] (a) -- (c);
    \draw[-latex, opacity = .2, gray] (b) -- (c);
    \end{scope}
    
    \begin{scope}[shift = {+(6,0)}]
     \node (a) at (0,0) {};
    \node (b) at (.5,.866) {};
    \node (c) at (1,0) {};
    
    \node at (0,0) [below left] {$ 0$};
    \node at (.5,.866) [above] {$ 1$};
    \node at (1,0) [below right] {$ 2$};

    \draw[black, fill] (a) circle (.035);
    \draw[black, fill] (b) circle (.035);
    \draw[black, fill] (c) circle (.035);
    
    \draw[-latex, opacity = .2, gray] (a) -- (b);
    \draw[-latex, opacity = .2, gray] (a) -- (c);
    \draw[-latex, bunired] (b) -- (c);
    \end{scope}

    \begin{scope}[shift = {+(9,0)}]
     \node (a) at (0,0) {};
    \node (b) at (.5,.866) {};
    \node (c) at (1,0) {};
    
    \node at (0,0) [below left] {$ 0$};
    \node at (.5,.866) [above] {$ 1$};
    \node at (1,0) [below right] {$ 2$};

    \draw[black, fill] (a) circle (.035);
    \draw[black, fill] (b) circle (.035);
    \draw[black, fill] (c) circle (.035);
    
    \draw[-latex, opacity = .2, gray] (a) -- (b);
    \draw[-latex, opacity = .2, gray] (a) -- (c);
    \draw[-latex, opacity = .2, gray] (b) -- (c);
    \end{scope}
    
    \end{tikzpicture}\\
    
    \begin{tikzpicture}[scale = 1.5]
    \node (a) at (0,0) {};
    \node (b) at (.5,.866) {};
    \node (c) at (1,0) {};
    \node (d) at (.5,0.288) {};

    \node at (0,0) [below left] {$ 0$};
    \node at (.5,.866) [above] {$ 1$};
    \node at (1,0) [below right] {$ 2$};
    \node at (.5,0.288) [below] {$ 3$};

    \draw[black, fill] (a) circle (.035);
    \draw[black, fill] (b) circle (.035);
    \draw[black, fill] (c) circle (.035);
    \draw[black, fill] (d) circle (.035);
    
    \draw[-latex, bunired] (a) -- (b);
    \draw[-latex, bunired] (a) -- (c);
    \draw[-latex, bunired] (b) -- (c);
    \draw[-latex, bunired] (a) -- (d);
    \draw[-latex, bunired] (c) -- (d);
    \draw[-latex, bunired] (b) -- (d);
    
     \begin{scope}[shift = {+(3,0)}]
     \node (a) at (0,0) {};
    \node (b) at (.5,.866) {};
    \node (c) at (1,0) {};
    \node (d) at (.5,0.288) {};

    \node at (0,0) [below left] {$ 0$};
    \node at (.5,.866) [above] {$ 1$};
    \node at (1,0) [below right] {$ 2$};
    \node at (.5,0.288) [below] {$ 3$};

    \draw[black, fill] (a) circle (.035);
    \draw[black, fill] (b) circle (.035);
    \draw[black, fill] (c) circle (.035);
    \draw[black, fill] (d) circle (.035);
    
    \draw[-latex, bunired] (a) -- (b);
    \draw[-latex, bunired] (a) -- (c);
    \draw[-latex, bunired] (b) -- (c);
    \draw[-latex, bunired] (a) -- (d);
    \draw[-latex, opacity = .2, gray] (c) -- (d);
    \draw[-latex, bunired] (b) -- (d);
    \end{scope}
     \begin{scope}[shift = {+(6,0)}]
     \node (a) at (0,0) {};
    \node (b) at (.5,.866) {};
    \node (c) at (1,0) {};
    \node (d) at (.5,0.288) {};

    \node at (0,0) [below left] {$ 0$};
    \node at (.5,.866) [above] {$ 1$};
    \node at (1,0) [below right] {$ 2$};
    \node at (.5,0.288) [below] {$ 3$};

    \draw[black, fill] (a) circle (.035);
    \draw[black, fill] (b) circle (.035);
    \draw[black, fill] (c) circle (.035);
    \draw[black, fill] (d) circle (.035);
    
    \draw[-latex, bunired] (a) -- (b);
    \draw[-latex, bunired] (a) -- (c);
    \draw[-latex, opacity = .2, gray] (b) -- (c);
    \draw[-latex, bunired] (a) -- (d);
    \draw[-latex, bunired] (c) -- (d);
    \draw[-latex, opacity = .2, gray] (b) -- (d);
    \end{scope}
    \begin{scope}[shift = {+(9,0)}]
     \node (a) at (0,0) {};
    \node (b) at (.5,.866) {};
    \node (c) at (1,0) {};
    \node (d) at (.5,0.288) {};

    \node at (0,0) [below left] {$ 0$};
    \node at (.5,.866) [above] {$ 1$};
    \node at (1,0) [below right] {$ 2$};
    \node at (.5,0.288) [below] {$ 3$};

    \draw[black, fill] (a) circle (.035);
    \draw[black, fill] (b) circle (.035);
    \draw[black, fill] (c) circle (.035);
    \draw[black, fill] (d) circle (.035);
    
    \draw[-latex, opacity = .2, gray] (a) -- (b);
    \draw[-latex, opacity = .2, gray] (a) -- (c);
    \draw[-latex, bunired] (b) -- (c);
    \draw[-latex, opacity = .2, gray] (a) -- (d);
    \draw[-latex, bunired] (c) -- (d);
    \draw[-latex, bunired] (b) -- (d);
    \end{scope}
    
     \begin{scope}[shift = {+(0,-2)}]
     \node (a) at (0,0) {};
    \node (b) at (.5,.866) {};
    \node (c) at (1,0) {};
    \node (d) at (.5,0.288) {};

    \node at (0,0) [below left] {$ 0$};
    \node at (.5,.866) [above] {$ 1$};
    \node at (1,0) [below right] {$ 2$};
    \node at (.5,0.288) [below] {$ 3$};

    \draw[black, fill] (a) circle (.035);
    \draw[black, fill] (b) circle (.035);
    \draw[black, fill] (c) circle (.035);
    \draw[black, fill] (d) circle (.035);
    
    \draw[-latex, bunired] (a) -- (b);
    \draw[-latex, bunired] (a) -- (c);
    \draw[-latex, opacity = .2, gray] (b) -- (c);
    \draw[-latex, bunired] (a) -- (d);
    \draw[-latex, opacity = .2, gray] (c) -- (d);
    \draw[-latex, opacity = .2, gray] (b) -- (d);
    \end{scope}
     \begin{scope}[shift = {+(3,-2)}]
     \node (a) at (0,0) {};
    \node (b) at (.5,.866) {};
    \node (c) at (1,0) {};
    \node (d) at (.5,0.288) {};

    \node at (0,0) [below left] {$ 0$};
    \node at (.5,.866) [above] {$ 1$};
    \node at (1,0) [below right] {$ 2$};
    \node at (.5,0.288) [below] {$ 3$};

    \draw[black, fill] (a) circle (.035);
    \draw[black, fill] (b) circle (.035);
    \draw[black, fill] (c) circle (.035);
    \draw[black, fill] (d) circle (.035);
    
    \draw[-latex, opacity = .2, gray] (a) -- (b);
    \draw[-latex, opacity = .2, gray] (a) -- (c);
    \draw[-latex, opacity = .2, gray] (b) -- (c);
    \draw[-latex, opacity = .2, gray] (a) -- (d);
    \draw[-latex, bunired] (c) -- (d);
    \draw[-latex, opacity = .2, gray] (b) -- (d);
    \end{scope}
    \begin{scope}[shift = {+(6,-2)}]
     \node (a) at (0,0) {};
    \node (b) at (.5,.866) {};
    \node (c) at (1,0) {};
    \node (d) at (.5,0.288) {};

    \node at (0,0) [below left] {$ 0$};
    \node at (.5,.866) [above] {$ 1$};
    \node at (1,0) [below right] {$ 2$};
    \node at (.5,0.288) [below] {$ 3$};

    \draw[black, fill] (a) circle (.035);
    \draw[black, fill] (b) circle (.035);
    \draw[black, fill] (c) circle (.035);
    \draw[black, fill] (d) circle (.035);
    
    \draw[-latex, opacity = .2, gray] (a) -- (b);
    \draw[-latex, opacity = .2, gray] (a) -- (c);
    \draw[-latex, bunired] (b) -- (c);
    \draw[-latex, opacity = .2, gray] (a) -- (d);
    \draw[-latex, opacity = .2, gray] (c) -- (d);
    \draw[-latex, bunired] (b) -- (d);
    \end{scope}
    \begin{scope}[shift = {+(9,-2)}]
     \node (a) at (0,0) {};
    \node (b) at (.5,.866) {};
    \node (c) at (1,0) {};
    \node (d) at (.5,0.288) {};

    \node at (0,0) [below left] {$ 0$};
    \node at (.5,.866) [above] {$ 1$};
    \node at (1,0) [below right] {$ 2$};
    \node at (.5,0.288) [below] {$ 3$};

    \draw[black, fill] (a) circle (.035);
    \draw[black, fill] (b) circle (.035);
    \draw[black, fill] (c) circle (.035);
    \draw[black, fill] (d) circle (.035);
    
    \draw[-latex, opacity = .2, gray] (a) -- (b);
    \draw[-latex, opacity = .2, gray] (a) -- (c);
    \draw[-latex, , opacity = .2, gray] (b) -- (c);
    \draw[-latex, opacity = .2, gray] (a) -- (d);
    \draw[-latex, opacity = .2, gray] (c) -- (d);
    \draw[-latex, , opacity = .2, gray] (b) -- (d);
    \end{scope}
    \end{tikzpicture}
    \caption{Small transitive tournaments and the corresponding incomplete tournaments}
    \label{fig:small_incoplete}
\end{figure}
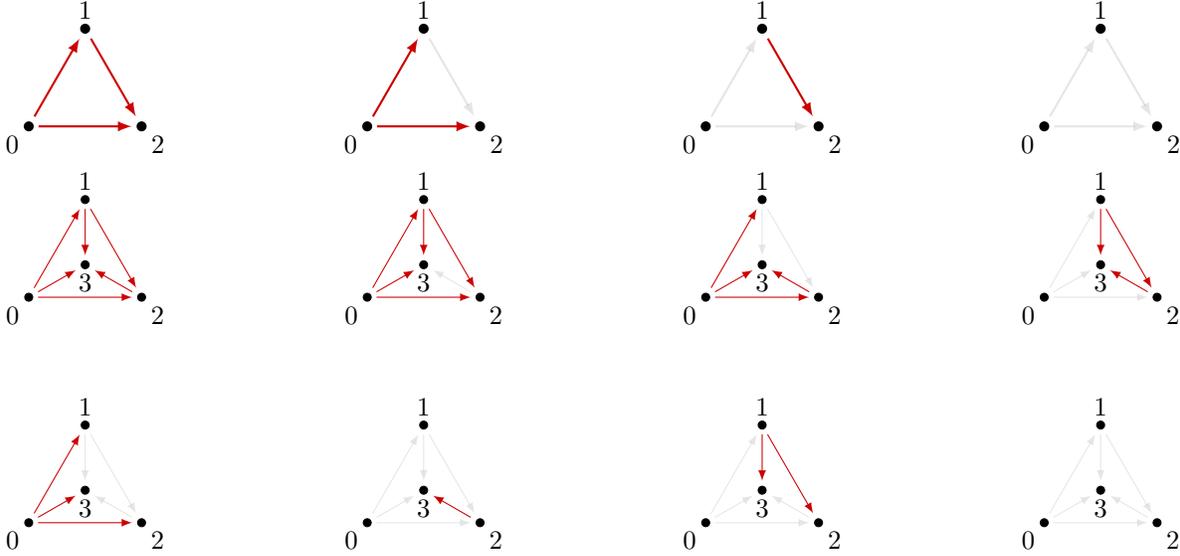

The proof of Theorem~\ref{thm:cliqueshomotopy} is now a straightforward application of the following lemma:

\begin{lem}\label{lem:inCompleteTour}
The multipath complex of all incomplete tournaments are empty, contractible, or a wedge of spheres.
\end{lem}

\begin{proof}
We proceed by induction, the cases $n=1,2, 3$ provided in Lemma~\ref{lem:SmallTournaments}.

Assume by induction that all incomplete tournaments in $\tT_h$, for $h\leq n$, are contractible or wedges of spheres. Let~$\tG$ be an incomplete tournament in $\tT_{n+1}$, say $\tG=\tT_{n+1}^{i_1, \dots, i_s}$. Without loss of generality, we can assume that $i_1< \dots < i_{s-1} < {n+1}$;  otherwise, $\tG$ is an incomplete tournament in $\tT_{n} \subseteq \tT_{n+1}$, in which case covered by the inductive assumption.  Observe that we can also assume that $i_1, \dots , i_{s-1} $ are not the full set $1,\dots , n$; otherwise $\tG$ would be a sink graph, hence its associated multipath complex would be a wedge of $0$-dimensional spheres. 

The strategy is to decompose $\tG$ in to smaller pieces as by Lemma~\ref{lem:BjornerLemma}. Let $\{j_0, \dots, j_{n-s}\} $ be the set $  \{ 0, 1 ,\dots, n\} \setminus \{ i_1,\dots,i_s\}$, with $j_0 < \dots < j_{n-s}$. Set $X_t$ to be the multipath complex associated to the subgraph~$\tG_t$ spanned by all edges which appear in a multipath featuring $(j_{n-s -t},n+1)$ in $\tG$  -- see also Figure~\ref{fig:decomposition5-3}. Observe that the simplicial complexes~$X_0,...,X_{n-s}$ cover $X(\tG)$. Furthermore, all the simplicial complexes $X_i$ are contractible; in fact, the edge~$(j_{n-s -t},n+1)$ is a module in $\tG_i$ (hence, $X_i$ is a cone). The intersection $X_i\cap X_j$ is contained in $X_0$: all multipaths which are both in~$\tG_i$ and~$\tG_j$ are multipaths in $\tG$ which do not feature the vertex $n+1$, and the vertex $j_{n-s}$ has outdegree $0$ in $\tG$ (and there are no oriented cycles in $\tT_{n+1}$). Therefore, by Lemma~\ref{lem:BjornerLemma}, the homotopy type of $X(\tG)$ is given by wedges of suspensions of $X_0\cap X_i$. To conclude, we want to show that $X_0\cap X_i$ is the multipath complex of an incomplete transitive tournament in $\tT_n$. This would conclude the proof by an inductive argument.

The complex $X_0\cap X_i$ is given by all multipaths in $\tG$ not featuring edges of type $(j_{n-s},p)$ and $(j_{n-s-i},q)$, for all $p$ and $q$, nor edges with target $n+1$. Hence, all such multipaths can be seen as multipaths in $\tT_{n}^I$ where $I$ is a re-ordering of the set $\{ i_1,...,i_s,j_{n -s - i},j_{n-s}\}$. \emph{Vice versa} all multipaths in $\tT_{n}^I$ appear as multipaths in $X_0\cap X_i$. Therefore, the complex $X_0\cap X_i$
can  be identified with the multipath complex of $\tT_{n}^I$, concluding the proof.
\end{proof}

\begin{rem}
Multipath complexes of transitive tournaments are generally not pure. In fact, computer-aided computations show that $\tT_7$ has non-trivial cohomology in degree $2$ and $3$, where ${\rm H}^2(X(\tT_7))\simeq \bZ^6$ and ${\rm H}^3(X(\tT_7))\simeq \bZ^{15}$. 
\end{rem}

\begin{rem}
It can be shown that $X(\tT_n)$ is shellable, and thus a wedge of spheres. Using a recursive coatom ordering, see \cite[Section~4.2]{Wac07}, where the coatoms of the top element (i.e. the maximal elements) are ordered lexicographically by their edges, and all other orderings follow canonically, since for every other element the downset is a Boolean lattice. However, we gain no new insight from this approach so omit the proof.
\end{rem}
If we consider the complete digraph $\tK_n$, where all edges are bidirectional, we no longer get wedges of spheres. In fact for $n=3$ the multipath complex $X(\tK_n)$ is 2 disconnected 1-spheres.

\bibliographystyle{alpha}
\bibliography{biblio}
\end{document}